\documentclass[11pt]{amsart}

\usepackage{amsmath,amsfonts,amsthm,amssymb,graphicx,tikz,color,url}
\usepackage[all,2cell,ps]{xy}

\theoremstyle{plain}
\newtheorem{thm}{Theorem}[section]
\newtheorem{lemma}[thm]{Lemma}
\newtheorem{prop}[thm]{Proposition}
\newtheorem{cor}[thm]{Corollary}
\newtheorem{conjec}[thm]{Conjecture}
\newtheorem{qtn}[thm]{Question}
\newtheorem{prob}[thm]{Problem}

\theoremstyle{definition}
\newtheorem{defn}[thm]{Definition}

\newtheorem{rem}[thm]{Remark}

\theoremstyle{remark}

\DeclareMathOperator{\SL}{SL} \DeclareMathOperator{\PSL}{PSL}
\DeclareMathOperator{\GL}{GL} 
 
 \DeclareMathOperator{\SO}{SO}

\DeclareMathOperator{\SP}{Sp}

 \DeclareMathOperator{\Ad}{Ad}

\DeclareMathOperator{\Diff}{Diff}\DeclareMathOperator{\sgn}{sgn}

\newcommand{\Fr}[1]{\ensuremath{\mathfrak{#1}}}

\newcommand{\fu}{\Fr u}
\newcommand{\fg}{\Fr g}

\newcommand{\fq}{\Fr q}

\newcommand{\BQ}{{\bf Q}}
\newcommand{\BR}{{\bf{R}}}
\newcommand{\BZ}{{\bf{Z}}}
\newcommand{\BC}{{\bf{C}}}
\newcommand{\T}{\ensuremath{{\bf T}}}

\title[$\mbox{SL}(n,\BR)$-actions on closed $n$-manifolds]{Smooth and analytic actions of $\mbox{SL}(n,\BR)$ and $\mbox{SL}(n,\BZ)$ on closed $n$-dimensional manifolds}
\author[D. Fisher]{David Fisher}
\address{Department of Mathematics\\Rice University\\Houston, TX 77005}
\email{davidfisher@rice.edu}

\author[K. Melnick]{Karin Melnick}
\address{Department of Mathematics\\University of Maryland\\College Park, MD, }
\email{karin@math.umd.edu}
\thanks{K. Melnick gratefully acknowledges support from NSF Award DMS-2109347.
D. Fisher gratefully acknowledges support from NSF Awards DMS-1906107 and DMS-2208430, the Miller Institute at Berkeley, IAS
and the Simons Foundation.  We thank Dave Witte Morris, Vincent Pecastaing, and Christian Rosendal for helpful conversations.}

\begin{document}

\begin{abstract}
The main theorem is a classification of smooth actions of
$\SL(n,\BR)$, $n \geq 3$, or connected groups locally isomorphic to it,
on closed $n$-manifolds, extending a theorem of Uchida
\cite{uchida.slnr.sn}.  We also construct new exotic actions of
$\SL(n,\BZ)$ on the $n$-torus and connected sums of $n$-tori, and we
formulate a conjectural classification of actions of lattices in
$\SL(n,\BR)$ on closed $n$-manifolds.  We prove some related results about
invariant rigid geometric structures for $\SL(n,\BR)$-actions.
\end{abstract}

\maketitle

\begin{center}
  \emph{In memory of Fuichi Uchida (1938--2021)}
  \end{center}

\section{Introduction}

\subsection{Classification of $\SL(n,\BR)$-actions}
\label{subsec:intro.classification}

Any smooth---even continuous,---faithful action  of $\SL(n,\BR)$ on an
$(n-1)$-dimensional manifold is the transitive action on ${\bf
  S}^{n-1}$, and the only other nontrivial action is the quotient action
  of $\PSL(n,\BR)$ on ${\bf RP}^{n-1}$. In 1979 F. Uchida constructed an infinite family of
real-analytic actions of $\SL(n,\BR)$ on ${\bf S}^n$ and proved his
construction yields all of them
\cite{uchida.slnr.sn}.  Previously, C.R. Schneider classified
all $C^\omega$ actions of $\SL(2,\BR)$ on closed surfaces and $\BR^2$
\cite{schneider.sl2r.surfaces}, see also \cite{stowe.sl2r.surfaces}.
A key role is played in both proofs by the linearizability theorem for
real-analytic actions of semisimple Lie groups on $(\BR^n, 0)$
due to Guillemin--Sternberg \cite{guillemin.sternberg.linearize} and
Kushnirenko \cite{kushnirenko.linearize}.  This theorem was partially
improved to $C^k$ linearizability of $C^k$ actions of $\SL(n,\BR)$ on
$(\BR^n, 0)$ for $k \geq 1$ and $n \geq 2$ by Cairns--Ghys
\cite{cairns.ghys.linearize}.  Relying partly on this result, we
classify smooth actions of $G$ on closed $n$-manifolds, where $G$ is
connected and locally
isomorphic to $\SL(n,\BR)$ and $n \geq 3$.

The actions are of two types---aside from a few exceptional transitive
actions in dimensions 3 and 4---, depending on the existence of
$G$-fixed points.  Let $Q < \SL(n,\BR)$ be the stablizer of a line in
the standard representation on $\BR^n$.  The actions without fixed
points are induced from $Q$ or $Q^0$-actions on $S^1$, yielding circle bundles
over ${\bf RP}^{n-1}$ or ${\bf S}^{n-1}$.  These are analogous to Schneider's actions on
${\bf T}^2$ or ${\bf K}^2$ for $n = 2$.
The actions with $G$-fixed points are actions on ${\bf S}^n$ or ${\bf RP}^n$, all
arising from the smooth version of Uchida's construction.  See
constructions I and II in Section \ref{subsec:Gactions} and the
classification theorems \ref{thm:no_fixed_points} and
\ref{thm:with_fixed_points} below.  Although the smooth linearization
theorem of Cairns--Ghys may permit a classification of smooth
$\SL(2,\BR)$-actions on surfaces, our classification will not be valid as
there is another family of actions in the case $n=2$.

A consequence of Theorems \ref{thm:no_fixed_points} and
\ref{thm:with_fixed_points} is that non-transitive real-analytic
actions can be parametrized
with real-analytic vector fields
on ${\bf S}^1$, which in turn are given by finitely-many real and
discrete parameters thanks to \cite{hitchin.vector.fields.s1},
along with some finite additional data; see Corollary
\ref{cor:analytic.paramzn} below.
This constitutes a \emph{smooth classification}, in the set-theoretic
sense (see \cite{rosendal.survey}), of analytic
$\SL(n,\BR)$-actions on closed $n$-manifolds, up to analytic conjugacy.

\subsection{Motivation from the Zimmer Program}

An important motivation for our classification are Zimmer's
conjectures on actions by semisimple Lie groups $G$, with no $\BR$-rank-one local
factors, and their
lattices on low-dimensional closed manifolds. See for example \cite{zimmer.icm,fisher.survey.festschrift, fisher.survey.update}.  The lowest possible
dimension for a nontrivial action of such a lattice should be the
minimal dimension $\alpha(G)$
of $G/Q$ where $Q$ is a maximal parabolic. For non-isometric, volume-preserving actions, the conjectured minimal dimension is the minimal
dimension $\rho(G)$ of a locally faithful
linear representation of $G$.  In general, the bound $\alpha(G) \leq
\rho(G) -1$ can have a significant gap; for $G= \SL(n,\BR)$, they are equal---that is $\alpha(G) =
n-1$ while $\rho(G)=n$.
For lattices in $\SL(n,\BR)$, $n
\geq 3$,  in joint work with A. Brown and S. Hurtado, the first author
proved both conjectured bounds
\cite{bfh.zimmer.conj, bfh.slmz, bfh.zimmer.nonuniform}.  They
moreover proved
dimension bounds for actions of lattices in
many other higher-rank
simple Lie groups; their results are sharp for
lattices in $\SL(n,\BR)$, $n \geq 3$, and $\SP(2n,\BR)$, $n \geq 2$.
These results resolved a major portion of
Zimmer's most famous conjecture.

Zimmer's Program asks more generally to what extent actions of
higher-rank semisimple Lie groups and their lattices on closed
manifolds arise from algebraic constructions.  The basic building
blocks of such constructions are actions on $G/H$ where $H$ is a
closed, cocompact subgroup, or actions of a lattice $\Gamma$ on
$N/\Lambda$, where $G$ acts by automorphisms of a nilpotent group $N$
and $\Gamma$ normalizes a cocompact lattice $\Lambda$.
Brown--Rodriguez-Hertz--Wang have announced a proof that any infinite action of a lattice $\Gamma < \SL(n,\BR)$ on
a closed $(n-1)$-manifold, where $n \geq 3$, extends to the standard action of $\SL(n,\BR)$ on ${\bf S}^{n-1}$ or
${\bf RP}^{n-1}$.

We propose to consider the above question for actions of $\SL(n,\BR)$ and its lattices on
closed $n$-manifolds, for $n \geq 3$. That this case is central to further progress is already emphasized
in \cite{fisher.survey.update}.  Theorems \ref{thm:no_fixed_points} and \ref{thm:with_fixed_points} fully
describe the $\Gamma$-actions that extend to $\SL(n,\BR)$.  The well-known action of the second type above, 
which does not extend to $\SL(n,\BR)$, is that of $\SL(n,\BZ)$ on ${\bf T}^n \cong
\BR^n/\BZ^n$.  In 1996 Katok--Lewis famously constructed exotic
$\SL(n,\BZ)$-actions on ${\bf T}^n$ in which the fixed point
corresponding to $ 0$ is blown up.  They also show that the weight determining the
action on the normal bundle of the blow-up can be freely chosen and that one
choice gives a volume-preserving exotic action \cite{katok.lewis.blowup}.

We construct new exotic actions of $\SL(n,\BZ)$, and its finite-index
subgroups, by gluing in ``exotic disks'' to ${\bf T}^n$ at $0$ or along
periodic orbits, and by forming connected sums of $n$-tori along
``exotic tubes.''  We conjecture the following classification of
actions of lattices in $\SL(n,\BR)$ on closed $n$-manifolds, for $n
\geq 3$. See Section \ref{subsec:lattice.actions} below for the relevant definitions.

\smallskip

{\bf Conjecture \ref{conj:gamma.actions}}
\emph{Let $\Gamma < \SL(n,\BR)$ be a lattice and $M$ a compact manifold of dimension $n$.
Then any action $\rho:\Gamma \rightarrow \Diff(M)$ either }
\begin{enumerate}
  \item \emph{extends to an action of $\SL(n,\BR)$ or $\widetilde{\SL(n,\BR)}$,}
  \item \emph{factors through a finite quotient of $\Gamma$, or}
  \item \emph{is an action built
from tori, $G$-tubes, $G$-disks, blow-ups, and two-sided blow-ups, with $\Gamma$ a
finite-index subgroup of $\SL(n,\BZ)$.}
\end{enumerate}

None of the actions in (1) and few of the actions in (3) are
volume-preserving (Proposition \ref{prop:vol.preserving}).  Conjecture \ref{conj:gamma.actions.volume} below asserts that
volume-preserving actions of $\Gamma$ as above are finite or are
actions of finite-index subgroups of $\SL(n,\BZ)$ built from tori,
blow-ups, and two-sided blow-ups with weight $n$ on the normal bundle,
as in \cite{katok.lewis.blowup}.

\subsection{Invariant rigid geometric structures}
\label{subsec:intro_rgs}

Largely inspired by Zimmer's results and conjectures, Gromov, together
with D'Ambra, proposed a program to investigate to what extent actions
of ``large''---for example, noncompact---Lie groups on closed
manifolds preserving a rigid geometric structure arise from algebraic
constructions (see \cite{gromov.rgs, dag.rgs}).  Benveniste--Fisher
proved that Katok--Lewis' actions do not preserve any rigid geometric
structure of algebraic type in the sense of Gromov
\cite{benveniste.fisher.no.rgs}.
We prove:

\smallskip

{\bf Theorem \ref{thm:no_projective}.}
\emph{Let $G$ be locally isomorphic to $\SL(n,\BR)$, acting smoothly on a compact $n$-manifold $M$,
preserving a projective structure $[ \nabla ]$.  Then $(M,[\nabla])$
is equivalent to}
\begin{itemize}
\item \emph{${\bf S}^n$ or ${\bf RP}^n$ with the standard projective
    structure; or }
\item \emph{a Hopf manifold, diffeomorphic to a flat
  circle bundle over ${\bf RP}^{n-1}$ with either trivial or ${\bf
    Z}_2$ monodromy. }
  \end{itemize}

  \smallskip

On the
other hand, we show in Proposition \ref{prop:invt.rgs} that all $\SL(n,\BR)$-actions on closed $n$-manifolds are
$2$-rigid in the sense of Gromov.

Pecastaing proved in \cite{pecastaing.lattice.global} that if a
uniform lattice in a simple Lie group $G$ of $\BR$-rank $\geq n$
admits an infinite action by projective
transformations of a closed $(n-1)$-manifold, then $G$ is locally
isomorphic to $\SL(n,\BR)$, and $\Gamma$ acts by the restriction of the standard
action on ${\bf S}^{n-1}$ or ${\bf RP}^{n-1}$.  Two interesting questions that
remain are:

\begin{qtn}
  \label{qtn:rgs.algebraic}
Are the
projective actions identified in Theorem \ref{thm:no_projective} the
only smooth actions of $\SL(n,\BR)$ on a closed $n$-manifold preserving
a rigid geometric structure of algebraic type?
\end{qtn}

\begin{qtn}
Given a non-affine, projective action of $\SL(n,\BZ)$ on a closed
  $n$-manifold, does it always extend to $\SL(n,\BR)$?
  \end{qtn}

\subsection{Other simple Lie groups.}

This work might be considered a special case of a more general
problem:

\begin{prob}
\label{problem:generalize}
For $G$ a simple Lie group of noncompact type, classify the smooth $G$-actions on compact manifolds of dimension $\alpha(G)+1$ up to smooth conjugacy.
\end{prob}

As above, $\alpha(G)$ is the minimal codimension of
a maximal parabolic subgroup of $G$.
This article concerns only groups locally isomorphic $\SL(n, \BR)$, $n
\geq 3$, because
this family is central in current research on the Zimmer Program, and
because the complexity of actions in
Problem \ref{problem:generalize} depends on the local isomorphism
type of $G$.

The only complete classification not mentioned so far is due to Uchida
for $G \cong \SL(n,\BC)$ \cite{UchidaComplex}.  In that case, there
are no faithful $G$-representations in dimension $2n-1= \alpha(G)+1$.
The actions in Uchida's classification correspondingly have no global
fixed points and are as in our Construction I below; these are induced
from actions of the maximal parabolic $Q$ or $Q^0$ on $S^1$.


We do not yet have a general conjectural picture for
all $G$.
Uchida has numerous
results for other simple groups \cite{UchidaSO1,UchidaSO2, UchidaSP,
  UchidaSurvey}, though none of these papers contains a complete
classification.
The only case in which we expect the classification to be more or less analogous to the one
presented here is for $G=\SP(2n,\BR)$, $n \geq 2$.
An interesting case is $G=\mbox{SO}(p,q)$, which has a faithful representation in
dimension $p+q = \alpha(G) +2$.  In the projectivization, the
stabilizers in open orbits are reductive, in contrast to what occurs
in our case (see Theorem \ref{thm:orbits}).
The other family of actions obtained by Schneider
\cite{schneider.sl2r.surfaces} referred to in Subsection
\ref{subsec:intro.classification} arises from the isomorphism
$\mbox{PSL}(2,\BR) \cong \mbox{SO}^0(1,2)$.

\section{Linearization and classification of orbit types}

A celebrated result of Guillemin--Sternberg \cite{guillemin.sternberg.linearize} and Kushnirenko \cite{kushnirenko.linearize}
states that a real-analytic
action of a semisimple Lie group $G$ on a real-analytic
manifold $M$ is
linearizable near any fixed point $p \in M$.  Linearization
means that there is a diffeomorphism $\Phi$ from a neighborhood $U$ of
$p$ to a neighborhood $V$ of $0 \in T_p M$ such that for all $g \in G$,
the germ of $g$ at $p$ equals the germ of $\Phi^{-1} \circ D_p g \circ
\Phi$.  Alternatively, for all vector fields $X$ arising from the $G$-action, $\Phi_*
X$ is the linear vector field $D_p X$ on $V \subset T_pM$.

Uchida's classification of analytic $\SL(n,\BR)$-actions on ${\bf S}^n$ for
$n \geq 3$ relies on this analytic linearization result.  Our
improvement to smooth actions is enabled by the smooth linearization
result of Cairns--Ghys \cite{cairns.ghys.linearize} for $\SL(n,\BR)$-actions on $\BR^n$
fixing $0$.  Also very useful for our arguments is a classification of
orbit types up to dimension $n$.  We recall their results in this section, together with
selected proofs.

\subsection{Classification of orbit types}
The following smooth orbit classification of \cite{cairns.ghys.linearize} will play a key
role in the sequel.

\begin{thm}[Cairns--Ghys \cite{cairns.ghys.linearize} Thm 3.5]
\label{thm:orbits}
Let $G$ be connected and locally isomorphic to $\SL(n,\BR)$ with $n\geq 3$, and assume $G$ acts continuously on a
topological manifold $M$. For any $x \in M$, the orbit $G.x$ is equivariantly
homeomorphic to one of the following:
\begin{enumerate}
  \item a point;
  \item ${\bf S}^{n-1}$ or ${\bf RP}^{n-1}$ with the projective action;
  \item $\BR^n \backslash\{0\}$ with the
    restricted linear action or
    $(\BR^n \backslash\{0\})/\Lambda$, for $\Lambda$ a discrete
    subgroup of the
    group of scalars $\BR^*$;
  \item one of the following closed, exceptional orbits, or a finite cover:
    \begin{itemize}
      \item For $n=3$, $\mathcal{F}_{1,2}^3$, the variety of complete flags in
        $\BR^3$
      \item For $n=4$, $\mbox{Gr}(2,4) = \mathcal{F}_2^4$, the Grassmannian of
        $2$-planes in $\BR^4$

        \end{itemize}
\end{enumerate}
\end{thm}

\begin{rem}
Some details:
\begin{itemize}
\item The actions in (2) are faithful for $\mbox{SL}(n,\BR)$ and
$\mbox{PSL}(n,\BR)$, respectively, while $\widetilde{\mbox{SL}(n,\BR)}$
does not act faithfully on any $(n-1)$-dimensional manifold.
\item Similarly, the
  actions in (3) are faithful for $\mbox{SL}(n,\BR)$ and
  $\mbox{PSL}(n,\BR)$, respectively, while $\widetilde{\SL(n,\BR)}$
  does not have a faithful $n$-dimensional representation.
\item The fundamental group of $\mathcal{F}_{1,2}^3$ is the quaternion
 group $Q_8$.  The universal
 cover is ${\bf S}^3$, on which $\widetilde{\SL(3,\BR)}$ acts faithfully.

 \item The fundamental group of
    $\mbox{Gr}(2,4)$ is $\BZ_2$ (see, \emph{e.g.}, \cite{gr24.topology}).  The universal cover is $S^2 \times
    S^2$, which can be identified with the space of oriented
    $2$-planes in $\BR^4$, on which $\mbox{SL}(4,\BR)$ acts faithfully.
    \end{itemize}
\end{rem}

To correct some oversights and provide additional details, we present
the proof here, more or less following the arguments of
\cite{cairns.ghys.linearize}.

\begin{proof}
An orbit $\mathcal{O}_x = G.x$ is a homogeneous space of $G$, identified with $G/G_x$, for $G_x \leq
G$ closed; thus the orbit is smooth with smooth $G$-action.  A maximal
compact subgroup $K$ is locally isomorphic to $\mbox{SO}(n)$, with dimension $n(n-1)/2$, and preserves a Riemannian
metric on $\mathcal{O}_x$.  By \cite[Thm II.3.1]{kobayashi.transf}, the isometry
group of an $m$-dimensional Riemannian manifold has dimension at most
$m(m+1)/2$, with equality if and only if it is ${\bf S}^m$ or ${\bf
  RP}^m$ with the standard $\mbox{SO}(m+1)$- or
$\mbox{PO}(m+1)$-action, respectively.  Thus any orbit of
dimension less than $n$ is as in parts (1) and (2) of the theorem; in
particular, all such orbits are closed.

Now assume that $\mathcal{O}_x$ is $n$-dimensional and consider $\fg_x \otimes
\BC$.  There is no reductive subalgebra of $\fg_{\BC} =
\mathfrak{sl}(n,\BC)$ of complex codimension less than or equal $n$.
Indeed, a suitable
Cartan decomposition $\mathfrak{k}_x + \mathfrak{p}_x$ of the
reductive subalgebra would align
with that of $\fg$, so the compact form $\mathfrak{k}_x + i
\mathfrak{p}_x$ would be contained in that of $\fg$.  By the same dimension
arguments as in the previous paragraph, there is no closed subgroup of
$\mbox{SU}(n)$ of codimension less than or equal $n$, assuming $n >
2$.  Thus there is no
compact
subalgebra of
$\mathfrak{su}(n)$ of codimension less than or equal $n$ for $n > 2$---a contradiction.

Assuming now that $\fg_x \otimes
\BC$ is not reductive,
 the isotropy
representation on $V = (\fg/\fg_x)_{\BC}$ will be reducible.  Assume there is
an invariant $p$-dimensional complex subspace, for $0 < p < n$.  The
stabilizer in $\mathfrak{sl}(n,\BC)$ of a $p$-dimensional subspace has
codimension $p(n-p)$, so $n \geq p(n-p)$.  Then $p=1$ or $n-1$ and
$\fg_x \otimes \BC$ has codimension $1$ in the subspace stabilizer, or
$n=4$ and $p=2$.

In the case $n=4$ and $p=2$, our dimension assumptions imply that
$\fg_x \otimes \BC$ equals the full stabilizer in
$\mathfrak{sl}(n,\BC)$ of
a $2$-dimensional complex subspace $W \subset V \cong
{\BC}^4$.  Because $\fg_x \otimes \BC$ does not preserve any proper
subspace of $W$, the intersection $W_0 = W \cap \overline{W} $ is
real-even-dimensional and $\fg_x$-invariant.  Since $W$ is assumed to
be a proper subspace, $W_0 \neq V_0$.

If $\mbox{dim } W_0 = 2$,
then $\fg_x$ is contained in the stabilizer of a $2$-dimensional subspace of
$\BR^4$.  As this stabilizer has real codimension $4$, it is equal to $\fg_x$.  The orbit $\mathcal{O}_x$ is the real Grassmannian $\mbox{Gr}(2,4)$
or a finite covering space.

The remaining possibility is that $W_0 = 0$.  This means $V = W
\oplus \overline{W}$.  Given $v_0 \in V_0$, there is a unique $w \in
W$ such that
$$ v_0 = \frac{1}{2}(w + \bar{w})$$
Then
$$ J(v_0) = \frac{i}{2}(\bar{w} - w)$$
defines a $\fg_x$-equivariant automorphism of $V_0$ with $J^2 =
-\mbox{Id}$.  Now
$\fg_x$ is contained in a
subalgebra isomorphic to $\mathfrak{sl}(2,\BC)$.  The codimension of
$\mathfrak{sl}(2,\BC)$ in $\mathfrak{sl}(4,\BR)$ is $9$, so this case
does not arise under our assumptions.

Next consider $p=1$ and $n \geq 3$.   For the invariant complex line
$W$, if $W \cap \overline{W} = 0$, then $\fg_x \otimes \BC$ preserves
a flag $W \subset U = W \oplus \overline{W} \subset \BC^n$.  The
stabilizer of such a flag has codimension $2n-3$, which is less than
or equal $n$ only for $n=3$.  In the case $n=3$, this flag is a full
flag, and the stabilizer has codimension $3$, so it equals $\fg_x
\otimes \BC$.  The intersection $U_0 = U \cap \overline{U}$ is
real-two-dimensional and $\fg_x$-invariant.  As in the previous
paragraph, the $\fg_x \otimes \BC$-invariant decomposition $U = W
\oplus \overline{W}$ defines an $\fg_x$-invariant complex structure on
$U_0$.  Now $\fg_x$ is contained in the stabilizer of a 2-plane $U_0$ in
$\BR^3$ together with a complex structure on $U_0$.  The codimension of
this stabilizer in $\mathfrak{sl}(3,\BR)$ is $4$.  Thus this case does
not arise under our assumptions.

Now we can assume $W_0 = W \cap \overline{W}$ is a real line, and $\fg_x$
has codimension $1$ in its stabilizer.  As in the previous paragraph,
$\fg_x$ must be irreducible on $V_0/W_0$ unless $n=3$, in which case,
if $\fg_x$ is reducible on this quotient, it is the Borel subalgebra
of $\mathfrak{sl}(3,\BR)$.  This case corresponds to $\mathcal{O}_x
\cong \mathcal{F}_{1,2}^3$ or a finite cover.

Now we assume $\fg_x$ is irreducible on $V_0/W_0$, and it is a
codimension-one subalgebra of the stabilizer $\mathfrak{q}$ of a line
in $\BR^n$.  The image of $Q$ on $V_0/W_0$ is
$\mbox{GL}(n-1,\BR)$, and $Q \cong \mbox{GL}(n-1,\BR) \ltimes
\BR^{n-1}$.  If the intersection $G_x \cap \BR^{n-1}$ were a
proper subspace, then $G_x$ would project onto $\mbox{GL}(n-1,\BR)$ by
dimension considerations.  But this would not be consistent with $G_x$ being
a subgroup, because $\mbox{GL}(n-1,\BR)$ is irreducible on $\BR^{n-1}$.  Therefore, $G_x$ has full intersection with this kernel
and projects onto a closed, codimension-one, irreducible subgroup of $\mbox{GL}(n-1,\BR)$.
Our earlier arguments show that $\mbox{SL}( n-1,\BR)$ has no closed,
codimension-one subgroup for $n \geq 4$, while for $n=3$, the unique
such subgroup is reducible.  Finally, we conclude that the projection
of $\fg_x$ to $\mathfrak{gl}(n-1,\BR)$ equals $\mathfrak{sl}(n-1,\BR)$, and $\fg_x
\cong \mathfrak{sl}(n-1,\BR) \ltimes \BR^{n-1} \lhd \mathfrak{q}$.
Let $E^0$ be the connected, normal subgroup of $Q$ isomorphic to
$\mbox{SL}(n-1,\BR) \ltimes \BR^{n-1}$, and $E_\Lambda \lhd Q$ the
inverse image of $\Lambda < Q/E^0 \cong \BR^*$.  The possibilities in
(3) correspond to $G_x = E^0$ or $E_\Lambda$, respectively, for
$\Lambda$ a nontrivial discrete subroup of $\BR^*$.

In the cases with $p=n-1$ and $n \geq 3$ the outer automorphism $g
\mapsto (g^{-1})^t$ gives an equivariant diffeomorphism from
$\mathcal{O}_x$ to one of the orbits in (3) or (4).
  \end{proof}

\subsection{Fixed Points and linearization}

Let $K < G$ be a maximal compact subgroup.  There is a $K$-invariant
Riemannian metric on $M$---it can be obtained by averaging any
Riemannian metric on $M$ over $K$ with respect to the Haar measure.
We will denote this metric $\kappa$.  The $K$-action
near a $K$-fixed point is linearizable, via the exponential map of $\kappa$.

\begin{prop}
  \label{prop:fixed_discrete}
  Let $G$ be connected and locally isomorphic to $\mbox{SL}(n,\BR)$.
For a nontrivial smooth action of $G$ on a connected manifold $M$ of dimension $n$, the fixed set
$\mbox{Fix}(G)$ is discrete.  In particular, if $M$ is compact, then
$\mbox{Fix}(G)$ is finite.
  \end{prop}

  \begin{proof}
Let $x \in M$ be a $G$-fixed point.  First suppose the isotropy
representation of $G$ at $x$ is trivial.  Then via the exponential map
of $\kappa$, we deduce that the $K$-action is trivial in a
neighborhood of $x$.  Then $K$ is trivial on all of $M$, and so is
$G$.

Now assume the isotropy representation of $G$ at $x$ is nontrivial;
then it is irreducible and factors through $\SL(n,\BR)$.  The
isotropy of $K$, which is isogeneous to $\mbox{SO}(n)$, is also
locally faithful and thus irreducible; in particular, there are no nontrivial fixed vectors.
Now by linearization of the $K$-action on a neighborhood, say, $U$, of
$x$, there are no
$K$-fixed points other than $x$ in $U$.  In particular, there are no
$G$-fixed points other than $x$ in $U$.
    \end{proof}

  We recall here the smooth linearization theorem for $\SL(n,\BR)$ of \cite{cairns.ghys.linearize}, exactly as stated there.

\begin{thm}[Cairns--Ghys \cite{cairns.ghys.linearize} Thm 1.1]
  \label{thm:linearization}
  For all $n > 1$ and for all $k=1, \ldots, \infty$, every
  $C^k$-action of $\SL(n,\BR)$ on $(\BR^n, 0)$ is $C^k$-linearizable.
  \end{thm}


There are two nontrivial $\SL(n,\BR)$-representations on $\BR^n$, the
standard one, which we will denote $\rho$,
and $\rho^*(g) = \rho((g^{-1})^t)$.  Under $\rho$, there is a
$Q$-invariant line, pointwise fixed by $E^0$, where $Q$ and $E^0$ are
the subgroups introduced in the proof of Theorem \ref{thm:orbits}.
Under $\rho^*$, there is no $Q$-invariant line when
$n \geq 3$; rather, $Q$ acts irreducibly on an invariant $(n-1)$-dimensional subspace.

\section{Constructions of smooth actions}
\label{sec:constructions}

In this section we construct all smooth, non-transitive, nontrivial
$\SL(n,\BR)$-actions on $n$-dimensional compact manifolds for $n \geq 3$.  The proof that the list is complete will
be given in the next section.  We will also give a conjecturally complete
description of smooth actions of lattices $\Gamma < \SL(n,\BR)$ on
compact $n$-dimensional manifolds.


Throughout
this section, $G=\SL(n,\BR)$ and $n \geq 3$.  The actions will be faithful or will
factor through faithful actions of $\PSL(n,\BR)$.
Recall that a maximal parabolic subgroup $Q < G$ is isomorphic to $\GL(n-1,\BR) \ltimes
\BR^{n-1}$. Let  $L \cong GL(n-1,\BR)$ be a Levi subgroup of $Q$ and $\pi:Q \rightarrow L$
the projection.   Let $\{ a^t \}$ be the one-parameter subgroup generating the identity component of the
center of $L$.  Let $C \cong O(n-1)$ be a maximal
compact subgroup of $L$.  Let $\sigma \in Z(C)$ project to $-1$ in $\BR^* \cong L/[L,L]$.
 Let $Q^0$ and $L^0$ be the identity components of $Q$ and
 $L$, respectively; note that $\pi(Q^0)=L^0$.  Define homomorphisms
\begin{eqnarray*}
  \nu_0: Q^0 \rightarrow \BR & \qquad & q \mapsto \ln (\det(\pi(q)) \\
\nu: Q \rightarrow \BR^* \cong \BZ_2 \times \BR & \qquad & q \mapsto
                                                            \left( \sgn(\det
                                                             (\pi(q)),
                                                             \ln
                                                             \left|
                                                            \det(\pi(q))
                                                            \right| \right)
 \end{eqnarray*}

The kernel of $\nu$ is $E_\Lambda$ with $\Lambda = \{ \pm 1\}$, which
will henceforth be denoted $E$; its identity
component $E^0$ is the kernel of $\nu_0$.

\subsection{Constructions of $G$-actions on closed $n$-manifolds}
\label{subsec:Gactions}

There are two families of non-transitive actions of $\SL(n,\BR)$ on
compact $n$-manifolds.  The first family have no $G$-fixed points and
are circle bundles over ${\bf RP}^{n-1}$ or ${\bf S}^{n-1}$.
Actions in the second family have two or one
$G$-fixed points and are diffeomorphic to ${\bf S}^n$ or ${\bf RP}^n$,
respectively.

 \subsubsection{Construction I: without global fixed points.} Let
 $\Sigma^0$ be a smooth circle.

 \begin{lemma}
   \label{lem:involution}
   If $\tau$ is a nontrivial smooth involution of $\Sigma^0$, then it has 0 or 2 fixed points.
 \end{lemma}

 \begin{proof}
   This is a fact of topology, but since our action is smooth,
  we will use the existence of a $\tau$-invariant metric.
  The fixed set of $\tau$ is closed and equals $\Sigma^0$ or is
finite.  Assuming $\tau$ is not trivial, the differential at these fixed
points is $-\mbox{Id}_1$.  Then if $\mbox{Fix}(\tau)$ is nonempty, the complement has
exactly two connected components, and $\mbox{Fix}(\tau)$ comprises
two points.
   \end{proof}

   Let $\{ \psi^t_X \}$ be a smooth
 flow on $\Sigma^0$, generated by a vector field $X$.
 Let $\tau$ be a smooth involution on $\Sigma^0$ commuting with $X$,
 or the involution of $\Sigma^0 \times \{1,-1 \}$ exchanging the two
components.  Let $\Sigma = \Sigma^0$ in the first case, and $\Sigma^0
\times \{ -1,1 \}$ in the second.  In the second case, extend $X$ to
$\Sigma$ by pushing forward via $\tau$ to the other component.

 Define an action of $\BR^*$ on $\Sigma$ by
 $$ t \mapsto \psi^{\ln |t|}_X \circ \tau^{(1- \mbox{sgn}(t))/2}$$
Via $\nu: Q \rightarrow \BR^*$, this lifts to an action of
$Q$ on $\Sigma$, which we will denote $\mu_{X,\tau}$.  Then
$$ M = G \times_Q \Sigma$$
is a closed manifold with smooth $G$-action.

Construction I with $\Sigma = \Sigma^0$, $X=0$, and $\tau = \mbox{Id}$, gives
the standard action on ${\bf RP}^{n-1}$ product the trivial action on
${\bf S}^1$, while $\Sigma = \Sigma^0
\times \{ -1,1\}$, $X=0$, and $\left. \tau \right|_{\Sigma^0}  =
\mbox{Id}$ gives the standard action on ${\bf S}^{n-1}$ product with the
trivial action on ${\bf S}^1$.  If $\left. \tau \right|_{\Sigma^0} =
\mbox{Id}$ and
$X$ is a constant nonvanishing vector field, then
$M$ is homogeneous and equivalent to $E_\Lambda$ as in
Theorem \ref{thm:orbits} (3), for
$\Lambda$ a lattice in $\BR^*$.
These are called Hopf manifolds and will be significant in
Theorem \ref{thm:no_projective} below.

   \subsubsection{Construction II: with global fixed points}
   \label{subsec:Gwithfixedpoint}

  Next we construct $\SL(n,\BR)$-actions on ${\bf S}^n$.  These
   are the same as those constructed by Uchida in \cite[Sec
   2]{uchida.slnr.sn}, except they are allowed to be only smooth.

   Let $\Sigma_+ = [-1,1]$.  Let $X$ be a
   smooth vector field on $\Sigma_+$ vanishing at $-1$ and $1$,
such that
   $D_{-1}X = 1 = D_{1} X$.  Notice that $X$ is nonvanishing
   on a nonempty open interval with $-1$ as endpoint, and similarly
   for $1$; moreover, $X$ has at least one zero in $(-1,1)$.
   Concretely, there is $z_- > -1$ the minimum of the zero set of $X$
   in $(-1,1)$ and $z_+< 1$ the maximum of the zero set of $X$ in $(-1,1)$.  It
   could be that $z_- = z_+$.

   Define a $Q^0$-action on $(-1,1)$ by letting $\{ a^t \}$ act by
   the flow $\{ \psi^t_X \}$ and then pulling back via the
   epimorphism $\nu_0 : Q^0 \rightarrow \BR^*_{>0}$.  Let $M' = G \times_{Q^0}
   (-1,1)$, a bundle over ${\bf S}^{n-1}$ with interval fibers.

  In $(\BR^n,0)$ with the standard action of $\SL(n,\BR)$, let $\ell_0$ be one of the two $Q^0$-invariant rays
   from the origin, pointwise fixed by $E^0$.  The restriction of $\{
   a^t \}$ to $\ell_0$ is smoothly equivalent to $\{ \psi^t_X \}$ on
   $(-1,z_-)$.  Identifying $\ell_0$ with $(-1,z_-)$ by this
   equivalence and extending $G$-equivariantly gives a smooth gluing of
   $(\BR^n,0)$ to $M'$, resulting in a manifold again diffeomorphic to
   $\BR^n$.  Similarly gluing another copy of $(\BR^n,0)$ along $\ell_0$
 to $(z_+,1)$ in a $Q^0$-equivariant way yields a closed manifold
   $M$, diffeomorphic to ${\bf S}^n$, on which $\SL(n,\BR)$ acts with
   two global fixed points.

   Last, we construct actions of $\mbox{SL}(n,\BR)$ on ${\bf RP}^n$.  Let
$\Sigma_{+} = [-1,0]$ and let $X$ be a smooth vector
field vanishing at $-1$ and $0$, such that $D_{-1} X = 1$.  As above,
define a $Q_0$-action on $\Sigma_{+}$ by composing the flow $\psi^t_X$
with $\nu_0$.
Let $M' = G \times_{Q^0} (-1,0]$.  As above, glue $\BR^n$ to
$M'$ by gluing $\ell_0$ to $(-1,z_-)$, for $z_- \leq 0$ the minimum of
the zeros of $X$ on $(-1,0]$.  The result is a manifold with boundary, diffeomorphic to ${\bf D}^n$.
The antipodal map corresponds to $[(g,x)] \mapsto [(g \sigma,x)]$,
which is well-defined because $\sigma \in Q$ and $\nu_0$ is invariant
under conjugation by $\sigma$.  The
$Q^0$-action on the $\Sigma_+$-fibers is equivariant with respect to
this involution.
Now quotient by the antipodal map restricted to the boundary of the
disk, mapping the boundary onto ${\bf RP}^{n-1}$.
The resulting space is diffeomorphic to
${\bf RP}^n$, with smooth, faithful $\SL(n,\BR)$-action.
Note that these actions on ${\bf RP}^n$ can be obtained as two-fold
quotients from actions
on ${\bf S}^n$ when $X$ is invariant under
$-\mbox{Id}_1$ on $\Sigma_+$.
 The standard $SL(n,\BR)$-representation $\rho$ on
$\BR^n$ product
with a one-dimensional trivial representation yields, after
projectivization, the ``standard action'' on ${\bf RP}^n$, obtained
from a standard embedding of $\SL(n,\BR)$ in $\SL(n+1,\BR)$.
This action is obtained from $X$ vanishing only at $-1$ and $0$, with
derivative $-1$ at $0$.  The standard action on ${\bf
  S}^n$ is the double cover, which arises from $X$ vanishing at
$-1,0,$ and $1$ only, with derivative $-1$ at $0$.

\subsection{New examples of lattice actions in supraminimal dimension}

The goal of this section is to develop new examples of actions of $\SL(n, \BZ)$ and its finite-index subgroups on manifolds
of dimension $n$, including new actions on the $n$-dimensional torus $\T^n$.  A sample result is

\begin{thm}
\label{thm:ourlatticetheorem}
Let $\Gamma \leq \SL(n,\BZ)$ be a finite-index subgroup. Then for $r>2$ and $r=\omega$ there
exist uncountably many $C^r$ actions of $\Gamma$ on $\T^n$, none of
which is $C^1$-conjugate to another.
\end{thm}

\subsubsection{Preliminaries: $G$-actions on blow-ups, disks and tubes}

From the constructions in the previous section, we will obtain exotic
$G$-actions on ${\bf D}^n$ and on ${\bf S}^{n-1} \times I$, for $I$ a
closed interval, which we will call \emph{$G$-disks} and
\emph{$G$-tubes}, respectively (Defs \ref{defn:gtube},
\ref{defn:gdisk}).  These will be patched into the standard action
on the torus to build more general actions than previously constructed.


 The blow-up of $\BR^n$ at the origin is constructed as the following algebraic
 subvariety of $\BR^n \times {\bf RP}^{n-1}$:
 $$B = \{ (x,[v]) \in  \BR^n \times {\bf
   RP}^{n-1} \ : \ x=cv \ \mbox{for some} \ c \in \BR \} $$
The $G$-action on $\BR^n \times {\bf RP}^{n-1}$ preserves $B$.
The projection onto the second coordinate exhibits $B$ as the tautological line bundle
over ${\bf RP}^{n-1}$.
In particular, $B$ is a manifold with an analytic $G$-action.  The
points of $B$ projecting to $0$ in the first factor form a
subvariety $E \cong {\bf RP}^{n-1}$, called the exceptional
divisor.

Let $BS$ be the universal cover of $B$.
The 2-to-1 covering $BS \rightarrow B$
is $G$-equivariant.
Note that although $BS$ is diffeomorphic to $\BR^n
\backslash \{0\}$, the $G$-action is not the restriction of
the linear action from $\BR^n$.
A construction of $BS$ analogous with that of $B$ is as follows: view ${\bf S}^{n-1}$ as
$\BR^n \backslash \{0\} / \BR^{*}_{>0}$ and define
 $$BS = \{ (x, [v]) \in \BR^n \times {\bf S}^{n-1} \ : \  x = cv \text{ for
   some }  c \in \BR  \}$$
As for $B$, the projection on the second factor exhibits $BS$ as a line bundle over ${\bf S}^{n-1}$.
 It is the tautological line bundle and is trivial.
 The covering $BS \rightarrow B$ corresponds to taking the quotient by
 the diagonal action of $-1$.  The subset of $BS$
 projecting to $ 0$ in the first factor will also be denoted
 $E$.  We define
 $$ BS^+ = \{ (x,[v]) \in BS \ : \ x = cv \ \mbox{for some } c \in \BR^*_{>0} \}$$
 and similarly for $BS^-$.

 Denote by $\ell_B \subset B$
 the fiber
 over the unique $Q$-fixed point in ${\bf RP}^{n-1}$.  The $Q$-action
 on $\ell_B$ factors through the homomorphism $\nu: Q \rightarrow \BR^*
 \cong \BZ_2 \times \BR$ and includes a flow, generated by
 a vector field $X_B$.  For $\{ a^t \} < Q$ the connected component
 of the center of $L$, the flow along $X_B$ is the $\{a^t\}$-action on $\ell_B$.
  For the linear $G$-action via $\rho$ on  $\BR^n$, let $\ell$ be the unique $Q$-invariant line,
 on which $Q$ acts via $\nu$.  The action of $\nu_0(Q^0)$ on $\ell$ is generated by a
vector field $X$ vanishing at $0$, which can be
 normalized so that $D_0 X =1$.
 Under the projection $B \rightarrow
 \BR^n$, the line $\ell_B$ is mapped to $\ell$.
 Up to
 normalization, we can assume that $D_0 X_B =1$.

 The $G$-action on $B$ is equivalent to the induced action on $G \times_Q \ell_B$.
 Similarly, on $BS$ there is a unique $Q^0$-invariant line
$\ell_{BS}$, and the $G$-action is the induced action on $G \times_{Q^0} \ell_{BS}$.
Moreover, $\ell_{BS}$ is $Q$-invariant. The deck transformations of
the covering $BS \rightarrow B$ can be realized as a $Q/Q^0$-action
commuting
with the $G$-action (as in the construction in the previous section of the action on ${\bf
  RP}^n$ as a quotient of an action on ${\bf D}^n$).

Modifying the vector field $X$ yields different $G$-actions on $B$ and
$BS$.  Given any $X$ on $\ell_B$ invariant under the $\BZ_2$-action and
vanishing only at $0$, the resulting $G$-space
$G \times_Q \ell_B$ will be denoted $B_X$.
The order of vanishing and derivatives of $X$ at $0$ can be arbitrary.
Similarly, any vector field $X$ on $\ell_{BS}$ vanishing only at $0$
yields a $G$-space diffeomorphic to $BS$, which we will denote $BS_X$.
If $X$ is moreover invariant under the $\BZ_2$-action on $\ell_{BS}$, there is
again a $G$-equivariant covering $BS_X \rightarrow B_X$.

The $G$-space $BS_X$, or in some cases just $BS_X^+$, will serve as a patch
between the standard torus action and the building blocks for our
exotic actions.  The blow-ups in Katok--Lewis' construction in
\cite{katok.lewis.blowup} are obtained by gluing a
space
$B_X$ into ${\bf T}^n \backslash \{0\}$.  They present $B_X$ and the
gluing in coordinates.  We will provide a coordinate-free construction
of their actions below.
Here are the building blocks for our exotic actions.

\begin{defn}
  \label{defn:gtube}
  Let $X$ be a vector field on $I = [-1,1]$ with $X(-1) = X(1) = 0$.
  Let $Q^0$ act on
  $I$ via $\nu_0$ followed by the flow along $X$.  The induced $G$-space $G
  \times_{Q^0} I$ is called a \emph{$G$-tube}.
\end{defn}

Now let $X$ be a vector field on $[-1,1]$ invariant by $-\mbox{Id}_1$, with
$D_0X=1$.  Let $\dot{D} = G \times_{Q^0} (0,1]$.  As in subsection
\ref{subsec:Gwithfixedpoint}, the linear $G$-action on $\BR^n$ can be
glued equivariantly onto an open subset of $\dot{D}$, yielding a $G$-action on ${\bf
  D}^n$.

\begin{defn}
  \label{defn:gdisk}
  These $G$-actions on ${\bf D}^n$ are called \emph{$G$-disks}.
  \end{defn}






\subsubsection{Lattice actions on compact manifolds}
\label{subsec:lattice.actions}

By gluing $G$-disks in place of fixed points in ${\bf T}^n$, we
will construct the examples of Theorem
\ref{thm:ourlatticetheorem} and prove the claim that they are $C^1$-distinct.
We will provide a common context for the famous examples of
Katok--Lewis \cite{katok.lewis.blowup} and our new examples, as well as an additional construction
using $G$-tubes to glue together tori.  The section finishes with a
conjecture on the classification of $\Gamma$-actions on closed
$n$-manifolds, for $\Gamma < \SL(n,\BR)$ a lattice.

For the standard action of $\Gamma = \SL(n,\BZ)$ on $\T^n$, local modifications
near the fixed point $0$
can be achieved by modifying the $\SL(n,\BZ)$-action on $\BR^n$
near $0$.
Given a fundamental domain $F \subset \BR^n$
containing the set
$U=(-\frac{1}{2}, \frac{1}{2})^n$, let, for each $\gamma \in \Gamma$,
$V_{\gamma}$ equal $\gamma^{-1}(U) \cap U \subset \BR^n$.  On $V_\gamma$ the covering map
$\pi: \BR^n \rightarrow \T^n$ is a diffeomorphism such that
$\pi(\gamma(x))=\gamma(\pi(x))$.  Changing the $\Gamma$-action on
$\BR^n$ on a neighborhood of $0$ contained in $U$ yields a
well-defined action on ${\bf T}^n$.
The same procedure is valid at any fixed point $p$ for $\Gamma$ or
a finite-index subgroup.
By such local modification a $G$-disk $D$ as in Definition \ref{defn:gdisk} can be glued
into the standard action ${\bf T}^n \backslash \{ 0 \}$, using a suitable $BS_X$ as the patch between
them.

To begin with, assume the vector field $X$ on
$[-1,1]$ determining $D$ vanishes at $1$ with derivative $1$ there.
Identify a collar neighborhood of $E$ in $BS^+$ radially with
a neighborhood of the puncture in $\dot{U} = \pi(U \backslash
\{  0 \})$ in ${\bf T}^n$, thus gluing $BS$ to the punctured torus.
The resulting space
inherits a well-defined, smooth $\Gamma$-action.  Next, identify a
collar neighborhood of $\partial D$ in $D$ radially with a collar
neighborhood of $E$ in $BS^-$, thus gluing $D$ to $BS$, while
retaining a smooth $\Gamma$-action.
The result of performing both gluings on one copy of $BS$ is a closed
manifold $M$ diffeomorphic to ${\bf T}^n$ with a well-defined
$\Gamma$-action. There is an invariant hypersphere corresponding to
$E$, such that the $\Gamma$-action in a neighborhood of this
hypersphere in $M$ is
equivalent to the $\Gamma$-action near $E$ in $BS$.  The
$\Gamma$-action is thus smooth in this neighborhood, and on
all of $M$.  In fact, if the $G$-action on $D$ is real-analytic, then
this gluing yields a $C^\omega$ action of $\Gamma$ on $M \cong {\bf T}^n$.

A general $G$-disk can be glued in to a punctured torus by the following procedure.
Let $h$ be a diffeomorphism of $\T^n\backslash\{ \pi( 0) \}$
that is the identity outside $\dot{U}$ and
is radial on its support.
Conjugate the $\SL(n,\BZ)$-action on $\T^n\backslash\{ \pi( 0) \}$ by
$h$.
Near the puncture, this action coincides with a modified $\SL(n,\BZ)$-action $\mu$ on
$\BR^n \backslash \{ 0 \}$.
Note that by Theorem \ref{thm:linearization}, the local linearization theorem, this action will not in most
cases extend to $\BR^n$.  It does, however, extend over a suitable $B_X$ or $BS_X$ patched into the puncture.  Indeed, the restriction of $Q^0$ to $\ell_0
\backslash \{ 0 \}$ corresponds to a vector field $X_0$ that extends to a vector field on
$\ell_0$ vanishing at the origin, which we will also denote by $X_0$. Identifying $\ell_0$ smoothly with
$\ell_B$ pushes $X_0$ forward to a vector field $X_B$ vanishing at $0$.  Then inducing
over $Q_0$ defines the $G$-action on $BS_X$.
The restriction to $BS_X^+$ is equivalent to $\mu$ near $0$, because
$h$ is radial.  Gluing $BS_X \backslash BS_X^-$ into $\BR^n \backslash \{ 0 \}$
along $BS_X^+$ gives a smooth $G$-action on $\BR^n$ with an open ball removed, which
is equivalent to $\mu$ on the complement of the boundary.  Making the local
identification with the torus, this gives a modified $\SL(n,\BZ)$-action on the torus minus an open
ball, equivalent to the originally modified action on the complement
of the boundary.  Now gluing a collar neighborhood of $\partial D$ into $BS^-_X$ exactly
as before yields an action on ${\bf T}^n$,
such that the $\Gamma$-action on $M \backslash D$ is equivalent to the
original action on $\T^n \backslash\{0\}$.

\begin{prop}
  \label{prop:G.disks.distinct}
Two actions of $\Gamma = \SL(n,\BZ)$ on ${\bf T}^n$ obtained by gluing a
$G$-disk in place of $\pi(0)$ as above are conjugate if and only if the vector
fields on $[-1,1]$ determining the $G$-disks are conjugate.
\end{prop}

\begin{proof}
Consider a $G$-disk $D$ determined by $\nu_0 : Q^0 \rightarrow
\BR$.
For a conjugate $\hat{Q}^0 < G$, there is a unique $\hat{Q}^0$-invariant
interval $\hat{I} \subset D$ on which the $\hat{Q}^0$-action is
given by
$\hat{\nu}_0: \hat{Q}^0 \rightarrow \BR$.  The given $G$-action on $D$
is the same as that induced from $\hat{Q}_0$ by $\hat{\nu}_0$.  We
will choose $\hat{Q}_0$ so that $\Gamma \cap \hat{Q}^0$ has dense image in $\BR$ under $\hat{\nu}_0$.
This implies that the $\Gamma$-action
determines the vector field $X$ on $\hat{I}$ for some, and hence any, choice
of conjugate $\hat{Q}^0$, which suffices to prove the proposition.

By \cite[Thm 1]{prasad.rapinchuk} of Prasad--Rapinchuk, there is a $\BQ$-irreducible
Cartan subgroup---often referred to as a $\BQ$-irreducible torus---in
$\Gamma$.  Denote it by $T$. Irreducibility here means that $T$ contains
no nontrivial, proper, algebraic subtorus.  There is a conjugate $\hat{Q}^0$ of
$Q^0$ in $G$ containing $T$.  It suffices to show that $\hat{\nu}_0$ is faithful
on $T_{\Gamma}= T \cap \Gamma$, since then the image will necessarily be dense in $\BR$.
If $\ker(\hat{\nu_0}) \cap T_{\Gamma}$ is nontrivial, let $T^0$ be its Zariski closure.  It is the kernel of the rational map $\hat{\nu}_0$ restricted to $T$.  Since $\Gamma$, and therefore $\ker(\hat{\nu_0}) \cap T_{\Gamma}$, consists of $\BZ$-points in $G$, $T^0$ is defined over $\BQ$.  It is thus an algebraic
subgroup of $G$, contained in $T$, contradicting $\BQ$-irreducibility.

\end{proof}

Here are additional constructions of $\Gamma$-actions, for $\Gamma =
\SL(n,\BZ)$ or a finite subgroup, on closed
$n$-manifolds.


{\bf Blow-up and two-sided blow-up.} For a vector field $X$ on
$\ell_{BS}$ invariant under the deck group of the cover $BS_X
\rightarrow B_X$, patching $BS_X$ into ${\bf T}^n \backslash \{ \pi(0)
\}$ and dividing by the deck group yields an $\SL(n,\BZ)$-action on ${\bf T}^n$ with $\pi(0)$ blown up.
Katok--Lewis' blow-up examples correspond to $X$ vanishing to first order at $0 \in \ell_{BS}$.
They computed that choosing $X$ with derivative $n$ at $0$ yields a volume-preserving action  \cite{katok.lewis.blowup}.

A related construction is what we will call a \emph{two-sided blow-up}.  (It is
discussed briefly in \cite{katok.lewis.blowup} and in more detail in \cite{benveniste.fisher.no.rgs, fisher.whyte.gd}.)
Start with two tori with $\pi(0)$ removed, $\dot{\bf T}^n_+$ and $\dot{\bf T}^n_-$.
Patch them together with $BS_X$ by gluing $BS^+_X$ into the punctured
neighborhood $\dot{U}_+$ and $BS^-_X$ into $\dot{U}_-$.  The resulting space is the connected sum of two tori with a smooth $\SL(n,\BZ)$-action.  A second variant involves a single torus punctured at two points which are fixed by
a finite-index subgroup $\Gamma < \SL(n,\BZ)$.  Patching neighborhoods of the two punctures together with $BS_X$ yields another smooth $\Gamma$-space.
Some, but not all, of these examples are equivariant covers of blow-up actions as constructed above, in which case the full $\SL(n,\BZ)$ acts on the cover.  As for the blow-ups, these actions are volume-preserving if $X$ has derviative $n$ at $0$.

{\bf Connected sum along $G$-tube. } Let $T$ be a $G$-tube with
defining vector field $X$.  Let $BS_X^a$
and $BS_X^b$ be two copies of $BS_X$ with exceptional divisors $E_a$ and $E_b$,
respectively.  For each of the
two boundary components of $T$, identify a collar neighborhood radially
with a collar neighborhood of $E_i$ in $(BS^i_X)^-$, for $i=a,b$, to glue $BS_X^a$ and $BS_X^b$
to $T$, one on each end.  Let ${\bf T}^n_a$ and ${\bf T}^n_b$ be two tori with $\pi(0)$ removed.
Then identify collar neighborhoods of $E_i$ in $(BS_X^i)^+$
radially with neighborhoods of the punctures in ${\bf T}^n_i$, for $i=a,b$, respectively.  The result is two punctured tori connected along the $G$-tube
$T$, with smooth (or even real-analytic) $\Gamma$-action.

{\bf Multiple $G$-disks along a finite orbit. }
Given a finite $\SL(n,\BZ)$-orbit $O$, a finite-index subgroup
$\Gamma$ fixes each point of $O$.  Gluing $G$-disks into some or all
of these $\Gamma$-fixed points by the procedure explicated above
yields additional $\Gamma$-actions on ${\bf T}^n$. One can also
perform  conjugate gluings along the periodic
orbit to obtain an $\SL(n,\BZ)$-action with multiple $G$-disks which
are permuted by the action.

{\bf Connecting points of a finite orbit by a $G$-tube. }  Given a
finite orbit $O$ as above, pointwise fixed by a finite-index
subgroup $\Gamma < \SL(n,\BZ)$, gluing $G$-tubes between some
pairs of distinct points of $O$ by the procedure above yields further
$\Gamma$-actions.

{\bf Combinations. }  Given a finite collection of tori $T_1, \ldots,
T_k$, each with finite orbits $O_i$, for $i = 1,
\ldots, k$, let $\Gamma <
\SL(n,\BZ)$
be a finite-index subgroup pointwise fixing $O = \cup_i O_i$.  Combinations of $G$-tubes and two-sided blow-ups between distinct points of $O$ and $G$-disks or blow-ups
at points of $O$ yield closed, connected $n$-manifolds with smooth $\Gamma$-action.


We refer to any action of a finite-index subgroup of $\SL(n, \BZ)$
constructed by finite iteration of the operations
described above as an \emph{action built from tori, $G$-disks, $G$-tubes,
blow-ups and two-sided blow-ups}.
Finite iterations of a subset of these operations yields a subset of these actions; for example, the actions in Theorem \ref{thm:ourlatticetheorem} are actions
  built from tori and $G$-disks.

\begin{conjec}
  \label{conj:gamma.actions}
Let $\Gamma < \SL(n,\BR)$ be a lattice and $M$ a compact manifold of dimension $n$.
Then any action $\rho:\Gamma \rightarrow \Diff(M)$ either
\begin{enumerate}
  \item extends to an action of $\SL(n,\BR)$ or $\widetilde{\SL(n,\BR)}$;
  \item factors through a finite quotient of $\Gamma$; or
  \item is an action built
from tori, $G$-tubes, $G$-disks, blow-ups and two-sided blow-ups, with $\Gamma$ a
finite-index subgroup of $\SL(n,\BZ)$
\end{enumerate}
\end{conjec}

Actions as in item $(1)$ are classified by Theorems
\ref{thm:no_fixed_points} and \ref{thm:with_fixed_points} below, so this conjecture amounts to a full description of $\Gamma$-actions in dimension
$n$.

We can formulate a much more restrictive conjecture for $\Gamma$-actions preserving a finite volume, thanks to the following proposition.

\begin{prop}
  \label{prop:vol.preserving}
Let $\Gamma \leq \SL(n,\BZ)$ be a finite-index subgroup acting on a closed manifold $M^n$ preserving a finite volume.  Then the $\Gamma$-action does not extend to $\SL(n,\BR)$ or $\widetilde{\SL(n,\BR)}$. If it is built from tori, $G$-disks, $G$-tubes, blow-ups, and two-sided blow-ups, then it is in fact built only from volume-preserving blow-ups and two-sided blow-ups.
  \end{prop}

Recall that the volume-preserving blow-ups and two-sided blow-ups are those with definiing vector field $X$ having derivative $n$ at $0$.

\begin{proof}
  Let $G \cong \SL(n,\BR)$.
For any lattice $\Gamma <G$, the $\Gamma$-action on $\BR^n \backslash \{0\}$ is ergodic by the Howe-Moore theorem (see \cite[Thm 2.2.6]{zimmer.etsg}).
Any volume form on $\BR^n \backslash \{0\}$ is $f \lambda$, for $\lambda$ the $G$-invariant Lebesgue measure and $f$ a smooth function. Thus the only $\Gamma$-invariant volume forms on $\BR^n \backslash \{0\}$ are constant multiples of $\lambda$, all
having infinite total volume.

Now let $\Gamma$ act on $M$ as in the statement of the proposition.  If the action is volume-preserving, there are no open sets on which the $\Gamma$-action is conjugate to the standard action on $\BR^n \backslash \{0\}$.  Any action extending to a non-transitive action of $\SL(n,\BR)$ or $\widetilde{\SL(n,\BR)}$, or any action containing $G$-tubes or $G$-disks, always contain such open sets.  The exceptional homogeneous spaces of Theorem \ref{thm:orbits} (4) do not have any $\Gamma$-invariant finite volume, also by the Howe-Moore Theorem.
\end{proof}

Here is the resulting conjecture for volume-preserving $\Gamma$-actions on closed $n$-manifolds:

\begin{conjec}
  \label{conj:gamma.actions.volume}
Let $\Gamma < \SL(n,\BR)$ be a lattice and $M$ a compact manifold of dimension $n$.
Then any action $\rho:\Gamma \rightarrow \Diff(M)$ either
\begin{enumerate}
  \item factors through a finite quotient of $\Gamma$; or
  \item $\Gamma \leq \SL(n,\BZ)$ is a
finite-index subgroup, and the action is built from tori and volume-preserving blow-ups and two-sided blow ups---that is, for which all vector
fields in the construction have derivative $n$ at $0$.
\end{enumerate}

\end{conjec}

\noindent While this version of the conjecture does not appear anywhere in the literature,
it seems to be widely believed by experts.  The more general Conjecture \ref{conj:gamma.actions}
is less established, mainly because the examples involving $G$-tubes and $G$-disks
were previously unkown.


\section{Classification of smooth $G$-actions}

Let $G$ be connected and locally isomorphic to $\SL(n,\BR)$.
In this section we prove that, aside from the
exceptional homogeneous spaces listed in Theorem \ref{thm:orbits}, all
nontrivial $G$-actions on closed $n$-manifolds are obtained from
constructions I or II from Section \ref{subsec:Gactions}.

It follows from Theorem \ref{thm:orbits} that a nontrivial
non-transitive action of $G$ is a faithful action of $\SL(n,\BR)$ or
$\mbox{PSL}(n,\BR)$.  We set $G = \SL(n,\BR)$ for the remainder of
this section and assume it acts nontrivially and non-transitively on a compact manifold $M$.

\subsection{Compact subgroups and fixed circle}

Let the subgroups $Q$, $L$, and $C$ be as in previous sections,
and let $C^0 \cong \SO(n-1)$.  As above, let $K$ be a maximal compact subgroup of $G$, containing $C$, and $\kappa$ a $K$-invariant metric on $M$.

\begin{prop}
  \label{prop:C0_fixed}
  Assume the $G$-action on $M$ is not transitive.
Let $\Sigma \subset M$ comprise the $C^0$-fixed points.  It is a nonempty, finite union of circles.
\end{prop}

\begin{proof}
  The $G$-orbits in Theorem \ref{thm:orbits} except those in (4) contain $C^0$-fixed points, and those in (4) are ruled out by our hypotheses.  Thus $\Sigma \neq \emptyset$.

By classical results, $\Sigma$ is a closed, totally geodesic
submanifold for $\kappa$.  It remains to verify that each connected
component has dimension $1$.  Let $x \in \Sigma$ and refer to Theorem
\ref{thm:orbits}.  If $x$ is a $G$-fixed point, then the $K$-action is
linearizable near $x$.  The isotropy representation of $K$ extends to $G$; it is the standard representation of $K$ on
$\BR^n$, in which the $C^0$-fixed set has dimension $1$.  If $\mathcal{O}_x$
has dimension $n-1$, then $C^0$ is irreducible on $T_x \mathcal{O}_x$
and trivial on the $\kappa$-orthogonal.  Then $\Sigma$ coincides with
the $\kappa$-geodesic orthogonal to $\mathcal{O}_x$ in a neighborhood
of $x$.  In the case $\mathcal{O}_x$ has dimension $n$, then by (3) of
Theorem \ref{thm:orbits}, the fixed set of $C^0$ in $\mathcal{O}_x$
is, as in the linear action on $\BR^n \backslash \{ {\bf 0} \}$, of dimension one.
\end{proof}

\subsection{Classification in the absence of $G$-fixed points}
\label{subsec:no_fixed_points}


\begin{thm}
  \label{thm:no_fixed_points}
  Let $G \cong \SL(n,\BR)$, acting non-trivially on a closed $n$-manifold $M$.  Assume that the $G$-action is not transitive and has no global fixed points.  Then the $G$-action on $M$ is as in Construction I---that is,
  $$ M = G \times_Q \Sigma$$
  where $Q$ acts via $\mu_{(X,\tau)}$, yielding one of the following:
    \begin{enumerate}
    \item 
      $M$ is
  diffeomorphic to ${\bf S}^{n-1} \times {\bf S}^1$, with faithful,
  fiber-preserving $G$-action.

\item 
  $M$ is diffeomorphic to ${\bf RP}^{n-1} \times {\bf S}^1$, with fiber-preserving action factoring through $\PSL(n,\BR)$.

\item 
  $M$ is a flat circle bundle with ${\bf Z}_2$ monodromy over ${\bf RP}^{n-1}$, with faithful $G$-action.

  \item 
    $M$ is diffeomorphic to the blow-up of ${\bf RP}^n$ at a point.  The
    $G$-action is faithful, leaves invariant the exceptional divisor and another
    hypersurface diffeomorphic to
    ${\bf RP}^{n-1}$, and preserves an ${\bf S}^{n-1}$-fibration
    on the complement of these two.
   \end{enumerate}
\end{thm}

\begin{proof}

  Our assumptions, together with Theorem \ref{thm:orbits}, imply that all point stabilizers are conjugate in $G$ into
$Q$.

Let $\Sigma$ be the fixed set of $C^0$, as in Proposition \ref{prop:C0_fixed}.
For $x \in \Sigma$, the stabilizer of $x$ contains $C^0$; denote this stabilizer
by $G_x$.  Let $h \in G$ be such that $h G_x h^{-1} \leq Q$; in
particular, $h C^0 h^{-1} \leq Q$.
The following homogeneous spaces are $K$-equivariantly diffeomorphic:
$$ G/Q \cong {\bf RP}^n \cong K/C$$
Now
$$ C^0 \leq \mbox{Stab}_G(hQ) \cap K = \mbox{Stab}_K(h'C)$$
for some $h' \in K$, where here the stabilizers are for the action by
translation on left-coset spaces.  Then $h' \in N_K(C^0) = C$.  It follows that
$h'C = C$ and thus $hQ = Q$.  We conclude that $G_x \leq Q$ for all $x
\in \Sigma$.

As the subgroups $E^0, E,$ and $Q^0$ are each normal in $Q$, the
stabilizer $G_x \in \{ E^0,E,Q^0,Q\}$ for all $x \in \Sigma$, using
Theorem \ref{thm:orbits}.  These stabilizers all contain $C^0$.  It
follows that $Q.\Sigma = \Sigma$. The normal subgroup $E^0$ is trivial
in restriction to $\Sigma$. Thus, the $Q$-action on $\Sigma$ factors
through the epimorphism $\nu: Q \rightarrow \BR^*$.

Let $\Sigma^0$ be a connected component of $\Sigma$.  Now $\BR^*_{>0}$
preserves $\Sigma^0$; this action is a smooth flow $\{ \psi^t_X \}$,
the restriction of $\{ a^t \}$.  Let $\sigma$ be an involution in $C$
mapping under $\nu$ to $-1$; it leaves $\Sigma$ invariant.  Let $\tau
= \left. \sigma \right|_\Sigma$.  Depending on $\tau$, let $\Sigma =
\Sigma^0 \times \{-1,1\}$ or $\Sigma = \Sigma^0$
Let $\mu_{(X,\tau)}$ be the corresponding $Q$-action on $\Sigma$.

For this $Q$-action on $\Sigma$, define the $G$-equivariant map
$$ \Phi: G \times_Q \Sigma \rightarrow M \qquad [(g,x)] \mapsto g.x$$
 The image of $\Phi$ is closed because the fiber product is compact.
Let $g.x$ be in the image of $\Phi$, with $g \in G$ and $x \in
\Sigma$.  If the orbit $G.x$ is $n$-dimensional, it is open; by
dimension comparison, the differential of $\Phi$ at $[(g,x)]$ is an
isomorphism.  If $G.x$ is $(n-1)$-dimensional, corresponding to $G_x = Q^0$ or $Q$, then $G.x$ is equivariantly diffeomorphic to ${\bf RP}^{n-1}$ or ${\bf S}^{n-1}$, and the $C^0$-fixed set in $G.x$ is
$0$-dimensional.  Thus $G.x$ is transverse to $\Sigma$ at $x$, and
the differential of $\Phi$ at $[(e,x)]$ is onto $T_{x} M$.  By
equivariance of $\Phi$, the differential at $[(g,x)]$ is also onto
$T_{g.x}M$.  We conclude that $\Phi$ is open, so it is a surjective local diffeomorphism---in this case, a covering map.

The Iwasawa Decomposition is a diffeomorphism
$$ K \times A \times N \rightarrow \SL(n,\BR)$$
where $K \cong \SO(n)$, as above, $A$ is the identity component of the diagonal subgroup, and $N$ is the group of unipotent upper-triangular matrices (see \cite[Thm VI.6.46]{knapp.lie.groups}).  As $N < E^0$ and $A/(A \cap E^0) \cong \{ a^t \}$, the Iwasawa Decomposition gives a normal form for elements of $G \times_Q \Sigma$: for any $g = k a' n \in G$ and $x \in \Sigma$,
$$ (g,x) \sim (k,a^t.x) \sim (k \sigma,\tau a^t.x)$$
where $a' = a^t a^{''}$ with $a^{''} \in A \cap E^0$.
Every $[(g,x)] \in G \times_Q \Sigma$ is represented by $(k,x)$ with $k \in K$ and $x \in \Sigma^0$.

If $\Phi([(k,x)] = \Phi([(k',x')])$ with $k,k' \in K$ and $x,x' \in
\Sigma^0$, then $k' = kq$ and $x' = q^{-1}.x$ for $q = k^{-1}k'$ in
the normalizer of $C^0$, which intersects $K$ in $C$.  Thus $[(k,x)] = [(k',x')]$.  We conclude that $\Phi$ is injective, hence a diffeomorphism.

Now suppose $\Sigma = \Sigma^0 \times \{ 1,-1 \}$.  Represent a point $p \in M$ by $[(k,x)]$ with $k \in K$ and $x \in \Sigma^0$. The assignment $p \mapsto (kC^0,x) \in K/C^0 \times \Sigma^0$ is well-defined, because the stabilizer of $\Sigma^0$ intersect $K$ equals $C^0$ in this case.  It is easy to verify that this map is a diffeomorphism $M \rightarrow {\bf S}^{n-1} \times {\bf S}^1$, corresponding to case (1).

If $\tau$ is trivial, then $M$ has a well-defined diffeomorphism to $K/C \times \Sigma^0 \cong {\bf RP}^{n-1} \times {\bf S}^1$, corresponding to case (2).

Next assume $\Sigma = \Sigma^0$ and $\tau$ acts freely.  In this case, the stabilizer in
$K$ of $\Sigma^0$ is $C$; note also that $C = \langle \sigma, C^0
\rangle$ and $\sigma$ normalizes $C^0$.  Given $p \in
M$ corresponding to $[(k,x)]$ with $k \in K$ and $x \in \Sigma^0$,
there is a well-defined map to the orbit $\{ (kC^0,x),
(k\sigma C^0,\tau.x) \} \in  (K/C^0 \times \Sigma^0)/\langle \sigma \rangle$.  This is case (3).

In the last case, when $\tau$ has two fixed points, say $x_0$ and
$x_1$, on $\Sigma^0$, then $\tau$ permutes the two components of $\Sigma^0 \backslash
\{ x_0, x_1 \}$.  Let $I_0$ be one component.  There is a well-defined map on $M \backslash ( G.x_0
\cup G.x_1)$ sending
$[(k,x)]$ to $(kC^0,x)$ with $x \in I_0$.  The image is diffeomorphic
to ${\bf S}^{n-1}
\times I_0$. The orbits $G.x_i$ are ${\bf RP}^{n-1}$. The manifold
$M$ can be obtained from ${\bf S}^{n-1} \times (\{ x_0 \} \cup I_o
\cup \{ x_1 \})$ by gluing ${\bf S}^{n-1} \times \{x_i\}$ to ${\bf
  RP}^{n-1} \times \{ x_i \}$ by the standard covering, for $i=0,1$. This is case (4).
\end{proof}

\subsection{Classification of actions with $G$-fixed points}

\begin{thm}
 \label{thm:with_fixed_points}
 Let $G \cong \SL(n,\BR)$, acting non-trivially on a closed
  $n$-manifold $M$.  Assume that the $G$-action is not transitive and
  has at least one global fixed point.  Then the $G$-action on $M$ is
  as in Construction II and has one or two fixed points.
  \begin{enumerate}
\item In the case
  of two fixed points, it is obtained from an induced action on $G \times_{Q^0}
  (-1,1)$ by attaching two copies of $\BR^n$ and is diffeomorphic to
  ${\bf S}^n$.
\item In the case of one fixed point, it is a two-fold
  quotient of an action as in (1), diffeomorphic to ${\bf RP}^n$.
\end{enumerate}
 In
  either case, the $G$-action is faithful.
  \end{thm}

\begin{proof}
Suppose that $x_0 \in M$ is $G$-fixed, and let $\Sigma^0$ be the
connected component of $\Sigma$ containing $x_0$.  This is a
$Q$-invariant curve through $x_0$, so there is a $Q$-invariant line
$\ell_0$ tangent to $\Sigma^0$ in
the isotropy representation of $G$ at $x_0$.  The isotropy is thus the
standard representation $\rho$.  Note also that $G \cong \SL(n,\BR)$.
Let $\{ a^t \}$, as above, be the one-parameter subgroup in the center of
$L \cong \GL(n-1,\BR) < Q$.  In a suitable parametrization $\rho(a^t)$
has eigenvalue $e^{t}$ on $\ell_0$.  Thus there
is a neighborhood of $x_0$ in $\Sigma^0$ in which $x_0$ is the only
$Q$-fixed point.  Let $\sigma \in C$ be as above,
so that $\rho(\sigma)$ acts as $-\mbox{Id}_1$ on $\ell_0$.  Both $\{ a^t \}$ and $\sigma$ have no fixed
points on $\Sigma^0 \backslash \{ x_0 \}$ in a neighborhood of $x_0$.
Thus in this neighborhood, points of
$\Sigma^0\backslash \{x_0 \}$ have stabilizer
contained in
$E^0$, which means, thanks to Theorem \ref{thm:orbits}, that these
stabilizers are $E^0$ and the corresponding orbits are $\BR^n
\backslash \{  0 \}$.  Then an $n$-dimensional $G$-orbit fills a punctured
neighborhood of $x_0$.
Finally, a neighborhood $U_0$ of $x_0$ is
$G$-equivariantly homeomorphic to $(\BR^n,0)$.

Now \cite[Thm 1.1]{cairns.ghys.linearize}, stated here as Theorem \ref{thm:linearization},
applies to give that the $G$-action on
$U_0$ is smoothly equivalent to the
representation $\rho$. Let $I_0 = \Sigma^0 \cap U_0$,
the open interval corresponding in these coordinates to the line $\ell_0$
through the origin pointwise fixed by $E^0$ and invariant by $Q$.

Let $\tau = \left. \sigma \right|_{\Sigma^0}$.  The involution $\tau$ has exactly one other fixed point, call it $x_1$, in
$\Sigma^0$, by Lemma \ref{lem:involution}.


\begin{prop}
  \label{prop:no_1pt_compactification}
The standard $\SL(n,\BR)$-representation $\rho$ on $\BR^n$ does not extend to a smooth action on any smooth one-point compactification.
  \end{prop}

\begin{proof}
  Assume $n \geq 3$.
Let $\{ a^t \}$ be the one-parameter subgroup as above, oriented such that $\|
\mbox{Ad } a^t \|  > 1$ on $\fu^+$, the unipotent radical of $\fq$, for $t > 0$.  This implies that $a^t$ is expanding on the $Q$-invariant line $\ell_0$ for $t > 0$.

Suppose that for a smooth structure on the one-point compactification
$\BR^n \cup \{ x_1 \}$ the $\SL(n,\BR)$-action extends smoothly.  In
the linearization at $x_1$ given by Theorem \ref{thm:linearization},
the image of $I_0 \cup \{x_1 \}$ contains a $Q$-invariant line
$\ell_1$.  Then the representation in this linearization is $\rho$.
That means $a^t$ is expanding on $\ell_1$ for $t > 0$.  Then the union
of the curves corresponding to $\ell_0$ and $\ell_1$ is a circle containing exactly two $\{a^t \}$-fixed points, both of which are expanding, a contradiction.

Though we do not need it here, we note the proof requires modification when $n=2$.  In that case,
direct computation shows that in $\rho^*$, the action of $a^t$ on the $Q$-invariant line also moves points away from the origin.  So when $n=2$, the contradiction is similar.
    \end{proof}


As $\{ a^t \}$ normalizes $C$ and commutes with $C^0$, it leaves
$\mbox{Fix} (\tau) = \{ x_0 , x_1 \}$ invariant.  Thus $x_1$ is also
$\{ a^t \}$-fixed.

\begin{cor}
At the $\tau$-fixed point $x_1$, the $\{a^t\}$-action is expanding on $\Sigma^0$.  The point $x_1$ does not lie on the boundary of $I_0$.
  \end{cor}

As in the proof above, the existence of a $Q$-invariant 1-manifold through $x_1$ forces the linearization at $x_1$ to be $\rho$, so $\{ a^t \}$ is expanding.  Since $\{ a^t \}$ is also expanding on $\Sigma^0$ at $x_0$, this precludes $x_1 \in \partial I_0$.

We now proceed with the identification of the action on $M$.  Let
$\Sigma^0_\pm$ be the two connected components of $\Sigma^0 \backslash
\{ x_0, x_1 \}$.  The stabilizers of all points of $\Sigma^0_+$ are
conjugate in $G$ to $Q, Q^0, E$, or $E^0$.  By the same argument as in
Section \ref{subsec:no_fixed_points}, the stabilizers are in fact equal to one of these subgroups.  One consequence is that $\Sigma^0_+$ is $Q^0$-invariant.
As $\tau.\Sigma^0_+ = \Sigma^0_-$, the union $\Sigma^0_+ \cup \Sigma^0_-$ is $Q$-invariant, and stabilizers of points in $\Sigma^0_+$ are in fact one of $Q^0$ or $E^0$.

Now we can define
$$ \Phi: G \times_{Q^0} \Sigma^0_+ \rightarrow M \qquad \Phi: [(g,x)] \mapsto g.x$$
As in the proof of Theorem \ref{thm:no_fixed_points}, $\Phi$ is a local diffeomorphism; as such, it has open image in $M$.

Recall that $G/Q^0 \cong K/C^0$.  There is in fact a natural $K$-equivariant diffeomorphism
$$ K \times_{C^0} \Sigma^0_+ \rightarrow G \times_{Q^0} \Sigma^0_+$$
mapping the $C^0$-orbit of $(k,x) \in K \times \Sigma^0_+$ to the
corresponding $Q^0$-orbit in $G \times \Sigma^0_+$. This map is
well-defined and injective because $C^0 = K \cap Q^0$.  It is easy to
see the map is open.  Surjectivity follows from the Iwasawa
Decomposition: write any $g \in G$ as a product $ka'n$, with $k \in
K$, $a' = a^t a^{''} \in A$, $a^{''} \in A \cap E^0$ and $n \in N < E^0$; then, given any $x \in \Sigma^0_+$, we have $[(g,x)] = [(k,a^t.x)]$.
The composition of this diffeomorphism with $\Phi$ is $[(k,x)] \mapsto
k.x$, which is injective because the stabilizer in $K$ of any $x \in \Sigma_+^0$ equals $C^0$.  We conclude that $\Phi$ is a diffeomorphism
onto its image, which is in turn diffeomorphic to ${\bf S}^{n-1} \times \Sigma^0_+$.

Let $(I_0)_\pm = U_0 \cap \Sigma^0_\pm$, so $(I_0)_+ \cup
(I_0)_- = I_0 \backslash \{ x_0 \}$.  The restriction of $\Phi$
to $G \times_{Q^0} (I_0)_+$ is a $G$-equivariant diffeomorphism to
$U_0 \backslash \{ x_0 \}$.  Under the $K$-equivariant identification
with ${\bf S}^{n-1} \times (I_0)_+$, the fibers $\{p \} \times (I_0)_+$ are $K$-equivariantly identified with the rays from the origin in $U_0 \cong \BR^n$.  Thus $U_0 \cup \mbox{Im } \Phi$ is $K$-equivariantly diffeomorphic to $\BR^n$.

Suppose $x_1$ is a $G$-fixed point, and let $U_1$ be an open
neighborhood of $x_1$ in $M$ on which the $G$-action is equivalent to
the linear action by $\rho$.  Let $(I_1)_+ = U_1 \cap \Sigma^0_+$.
As in the previous paragraph, $\Phi$ restricted to $G \times_{Q^0}
(I_1)_+ \cong {\bf S}^{n-1} \times (I_1)_+$ identifies fibers
$\{ p \} \times (I_1)_+$ with rays from the origin in $U_1
\backslash \{ x_1 \}$ in a $K$-equivariant manner.  The fibers $ \{ p
\} \times (I_1)_+$ are in turn identified $K$-equivariantly with
infinite segments of rays from the origin in $\BR^n$ under its
identification with $U_0 \cup \mbox{Im } \Phi$.  Thus $U_1 \cup
\mbox{Im} \Phi \cup U_0$ is $K$-equivariantly diffeomorphic to ${\bf
  S}^n$.  It is moreover open and closed in $M$, so it equals $M$.  We
conclude that $M$ is $G$-equivariantly diffeomorphic to the $G$-action
on ${\bf S}^n$ in Construction II with $ \{ \psi^t_X \}$ equal $\{ a^t
\}$ restricted to $\Sigma^0_+$.

Now suppose $x_1$ is not $G$-fixed. The stabilizer of $x_1$ contains
$C^0, \sigma$, and $\{ a^t \}$.  Thus it equals $Q$, and the orbit of $x_1$ is
$G/Q \cong K/C \cong {\bf RP}^{n-1}$. The $K$-invariant metric $\kappa$
determines a normal bundle to $G.x_1$, and an identification of a
neighborhood of the zero section with a normal neighborhood $U_1 \cong
K/C \times (-\epsilon, \epsilon)$ of the orbit.  The fiber over $x_1$,
call it $I_1$, comprises $C^0$-fixed points and is contained in
$\Sigma^0$.  Thus $I_1 \backslash \{ x_1 \}$ intersects $U_0 \cup
\mbox{Im } \Phi$ in two components, $(I_1)_+$ and $(I_1)_-$,
contained in $\Sigma^0_+$ and $\Sigma^0_-$, respectively.  The
saturation $K.(I_1)_+$ equals $U_1 \backslash G.x_1$, and each
distinct translate $k.(I_1)_+$ is identified with an infinite
segment of a unique ray from the origin in $U_0 \cup \mbox{Im } \Phi
\cong {\BR}^n$.  The resulting $K$-equivariant gluing of the normal
bundle of $G.x_1$ to $U_0 \cup \mbox{Im } \Phi$ is equivalent to the
gluing of the normal bundle of ${\bf RP}^{n-1}$ to ${\BR}^n$ yielding
${\bf RP}^{n}$.  We obtain that $M = U_0 \cup \mbox{Im } \Phi \cup
U_1$ is diffeomorphic to ${\bf RP}^n$, with the $G$-action on ${\bf
  RP}^n$ in
Construction II corresponding to $\{ \psi^t_X \}$ equal $\{ a^t \}$
restricted to $\Sigma_+^0 \cup \{ x_1\}$.
\end{proof}

\section{Analytic classification}

In construction I of section \ref{subsec:Gactions}, the
actions are determined by the vector field $X$ on $\Sigma^0 \cong {\bf
  S}^1$ and the involution $\tau$ of $\Sigma$ commuting with $X$.
There are four possibities for $\tau$, corresponding to the four
possible diffeomorphism types in Theorem \ref{thm:no_fixed_points}.
In construction II, the action is determined by the vector field $X$
on the interval
$\Sigma_+ = [-1,1]$; the ${\bf RP}^n$-actions correspond to $X$ being
invariant by $x \mapsto -x$.  By doubling $\Sigma_+$ and gluing at the
endpoints $-1$ and $1$, the vector field in this case determines a
vector field on ${\bf S}^1$ invariant by a reflection (invariant by
two reflections in the case corresponding to an action on ${\bf
  RP}^n$).

Thus, aside from the aforementioned finite data,
$G$-actions on
closed $n$-manifolds are determined by a smooth vector field on a
circle, with some additional symmetries according to the type and
subtype.

\begin{prop}
  \label{prop:analyticity}
The vector field $X$ and the involution $\tau$ are real-analytic,
rather than just smooth, if and only if the resulting $n$-manifold and
$\SL(n,\BR)$-action are real analytic.
\end{prop}

\begin{proof}
Let $G = \SL(n,\BR)$.
  First assume $M = G \times_Q \Sigma$ as in Construction I.  If the
  vector field $X$ and involution $\tau$ are $C^\omega$, then the
  resulting $\BR^*$-action on $\Sigma^0$ is $C^\omega$.  As $\nu : Q
  \rightarrow \BR^*$ is a $C^\omega$ homomorphism, the lifted
  $Q$-action on $\Sigma^0$ is $C^\omega$.  Next, $Q < G$ is an
  analytic---in fact, algebraic---subgroup, so the diagonal $Q$-action
  on $G \times \Sigma^0$ is $C^\omega$.  We conclude that $M$, the
  quotient by this action, is $C^\omega$.

  If $M$ is built from a $C^\omega$ vector field $X$ on $\Sigma_+ =
  [-1,1]$, then the resulting $Q^0$-action is $C^\omega$ on $(-1,1)$,
  so $M' = G \times_{Q^0} (-1,1)$ is $C^\omega$.  The diffeomorphisms
  from $\ell_0$ to $(-1,z_-)$ and $(z_+,1)$ conjugating $\{ a^t \}$ to the
  respective restrictions of $\{ \psi^t_X \}$ are then $C^\omega$, as
  are their $G$-equivariant extensions.  The gluings are then
  $C^\omega$ quotient maps, so the resulting action on ${\bf S}^n$, or
  the two-fold quotient, ${\bf RP}^n$, is $C^\omega$.  (The analyticity
  in this paragraph was previously proved by Uchida
  \cite[Sec 2]{uchida.slnr.sn}.)

Now suppose $M^n$ is $C^\omega$ with real-analytic $G$-action.
We are assuming the $G$-action is not transitive.  The compact
subgroup $C^0 < G$ is analytic, so the fixed set $\Sigma$ is, too.
As shown in the proofs of Theorems
\ref{thm:no_fixed_points} and \ref{thm:with_fixed_points},
$\Sigma$ is $Q$-invariant.  The restriction of the one-parameter
subgroup $\{a^t \}$ to any component of $\Sigma$ is $C^\omega$.  Similarly, the restriction of the
involution $\sigma \in C$ to any component is $C^\omega$.  These yield the vector
field $X$ and the involution $\tau$, respectively, so this data is real-analytic.
  \end{proof}

N. Hitchin gave a complete set of invariants for $C^\omega$ vector
fields on $S^1$ in \cite[Thm 3.1]{hitchin.vector.fields.s1}.  They are as
follows, for $X \in \mathcal{X}^\omega(S^1)$:
\begin{itemize}
\item A nonnegative integer \emph{number} $k \in {\bf N}$ \emph{of zeros} of $X$.

  \smallskip
  Given a choice of cyclic ordering of the zeros,
  \smallskip

\item An \emph{orientation} $\sigma \in \{ \pm 1 \}$.
  \item A list $m_1, \ldots, m_k$ of positive integers, the
    \emph{orders of vanishing} of $X$ at each zero
    \item A list $r_1, \ldots, r_k$ of real numbers, the
      \emph{residues} of $X$ at each zero.  When $X$ vanishes to order
      1 at $x_i$, the residue $r_i$ is the reciprocal of the
      derivative of $X$ at $x_i$.  The residues are defined analytically in
      general, see \cite[Sec 1]{hitchin.vector.fields.s1}.
      \item  A \emph{global invariant} $\mu \in \BR$.  For $X = f \partial \theta$ nonvanishing, this is the integral around
        $S^1$ of $d \theta/f$.  For general $f$ it is analytically
        defined, see \cite[Sec 2]{hitchin.vector.fields.s1}.
\end{itemize}

Different choices of orientation and cyclic ordering of the zeros
correspond to the dihedral group $D_k$.
More precisely,
$$ \left( \{ \pm 1 \} \times \BR \times \bigsqcup_{k=0}^\infty ({\bf N}^k \times \BR^k) \right) / \bigsqcup_{k=0}^\infty D_k $$
is a Borel subset of a Polish space, providing a \emph{smooth
  classification} of analytic vector fields on $S^1$ up to analytic
conjugacy (see \cite{rosendal.survey, Foreman} for background on this
set-theoretic notion).

  \begin{cor}
    \label{cor:analytic.paramzn}
Real-analytic actions of $\SL(n,\BR)$ on closed, analytic
$n$-manifolds are classified up to equivariant, real-analytic diffeomorphism by the
following set of invariants:
\begin{enumerate}
  \item A type, I or II, or one of the finitely-many transitive
    actions in Theorem \ref{thm:orbits}.
  \item For type I, one of four possible conjugacy classes for the
    analytic involution $\tau$,
    and Hitchin's set of invariants for the analytic vector field $X$,
    commuting with $\tau$.
\item For type II, one of two homotopy types of $M$ and Hitchin's set
  of invariants for the analytic vector field $X$ on $S^1$, having at
  least two zeros of order one with identical positive residues, invariant
  by reflection in this pair of zeros, and
  additionally invariant by the antipodal map in the case $M$ is
  not simply connected.
    \end{enumerate}
    \end{cor}

    The assignment of types and conjugacy classes of suitable vector fields on
    $S^1$ to $\SL(n,\BR)$-actions on closed
    $n$-manifolds, up to equivariant diffeomorphism (in the smooth or
    $C^\omega$ category), is a
    \emph{Borel reduction} (again, see \cite{rosendal.survey, Foreman}); recall, this assignment is by the restriction of
    a certain one-parameter subgroup $\{ a^t \} < \SL(n,\BR)$ to a
    component $\Sigma^0$ of the fixed set of $C^0 \cong
    \mbox{SO}(n-1)$.  The reverse assignment, starting from a type and
    a  conjugacy class of compatible vector fields and constructing a
    closed $n$-manifold with $\SL(n,\BR)$-action, up to equivariant
    diffeomorphism, is also a Borel reduction.  Our construction and classification
    result give a \emph{Borel bireduction} between these two
    equivalence relations, whether in the smooth or analytic category.
Thanks to Hitchin's result we obtain in the analytic case a smooth
classification (in the set-theoretic sense above) of analytic
$SL(n,\BR)$-actions on closed manifolds of dimension $n$ up to
analytic conjugacy.
The subject of smooth classification in dynamical systems has recently received considerable
attention.

Smooth vector fields, in contrast to
analytic ones, do not admit a smooth classification: $E_0$, a
particular Borel equivalence relation on a standard Borel space
known not to admit a smooth classification, can be Borel reduced to
it.  It follows that $\SL(n,\BR)$-actions on closed $n$-manifolds in
the smooth category do not admit a smooth classification (again, see \cite{rosendal.survey, Foreman} for
these notions and the definition of $E_0$).  We thank Christian
Rosendal for very helpful conversations
on this topic.

    \section{Invariant Geometric Structures}
The linear action of $\SL(n,\BR)$ on $\BR^n$ preserves
the standard, flat affine structure, while the transitive action on ${\bf
  S}^{n-1}$ preserves the standard, flat projective structure.

A \emph{projective structure} is an equivalence class of torsion-free
connections, where $\nabla \sim \nabla'$ means there is a $1$-form
$\omega$ on $M$ such that
$$ \nabla'_X Y = \nabla_X Y + \omega(X) Y + \omega(Y) X$$
for all $X,Y \in \mathcal{X}(M)$.
Equivalent connections determine the same geodesic curves, up to
reparametrization. See \cite[Ch 8]{sharpe} or \cite[Ch IV]{kobayashi.transf}.

All actions of $\SL(n,\BR)$ on closed $n$-manifolds are built from
projective actions on $\BR^n$,
$\BR^n \backslash \{ 0 \}$, and ${\bf S}^{n-1}$, but only a few
well-known examples preserve a projective structure.  We will prove this in Section \ref{subsec:no_projective}
below.  These actions all do, however, preserve a \emph{rigid geometric
structure of order two}, a much more flexible notion due to Gromov.

    \subsection{Invariant $2$-rigid geometric structure}

For $k \geq 0$, denote by $\mathcal{F}^{(k)} M$ the bundle of
$k$-frames on $M$ of order $k$, with $\mathcal{F}^{(0)} M = M$.  A
$k$-frame at $x \in M$ is the $k$-jet at $0$ of a coordinate chart
$(\BR^n,0) \rightarrow (M,x)$.  These form a principal
$\GL^{(k)}(n,\BR)$-bundle, where this is the group of $k$-jets at $0$
of local diffeomorphisms of $\BR^n$ fixing $0$.

\begin{prop}
  \label{prop:invt.rgs}
Given any nontrivial, smooth action of $\SL(n,\BR)$ on a compact, $n$-dimensional
manifold $M$, the action of $\SL(n,\BR)$ on $\mathcal{F}^{(2)}M$ is
free and proper.  In particular, the action preserves a $2$-rigid
geometric structure in the sense of Gromov.
  \end{prop}

For the definition of \emph{rigid geometric structure of order $k$}
we refer to \cite[0.3]{gromov.rgs}, \cite[Def 3.7]{ballmann.rgs}, or
\cite[Sec 4]{feres.framing.frobenius}.  For the equivalence for a smooth
Lie group action between preserving a $k$-rigid geometric structure
and acting freely and properly on
$\mathcal{F}^{(k)}M$, see
\cite[0.4]{gromov.rgs} or \cite[Thm 3.22]{ballmann.rgs}

\begin{rem}
  Gromov asserts in \cite[0.4.C3]{gromov.rgs} that any real-analytic action
  of a semisimple Lie group with finite center is 3-rigid.
  Benveniste-Fisher assert the 2-rigidity of a specific
$\SL(n,\BR)$-action on a manifold of type (4) in
Theorem \ref{thm:no_fixed_points}, obtained by blowing up the origin in
the standard $\SL(n,\BR)$-action on ${\bf RP}^n$ \cite[Sec 3]{benveniste.fisher.no.rgs}.
\end{rem}

\begin{lemma}
  \label{lem:free.proper.submersion}
Let $\pi: M' \rightarrow N$ be a smooth submersion and $k \geq 0$.  Suppose $\pi$ is
equivariant with respect to smooth actions of a group $G$ on $M'$ and
$N$.  If $G$ acts freely and properly on $\mathcal{F}^{(k)}N$, then it
acts freely and properly on $\mathcal{F}^{(k)}M'$.
\end{lemma}

\begin{proof}
  Let $m = \dim M'$ and $n = \dim N$.
Let $\mathcal{S} \subset \mathcal{F}^{(k)}M'$ comprise the $k$-jets of
coordinate charts $\varphi: (\BR^m,0) \rightarrow (M',x)$ for which
$(\pi \circ \varphi)(0 \times \BR^{m-n})$ is constant to order $k$ at $0$, for all $x
\in M'$; in other words, $\varphi(0 \times \BR^{m-n})$ is tangent to
the $\pi$-fiber of $x$ to order $k$. The set $\mathcal{S}$ is
$G$-invariant and closed---in fact, it is a subbundle of $\mathcal{F}^{(k)}M'$.

Each $k$-frame in $\mathcal{S}_x$ gives a $k$-frame to $N$ at
$\pi(x)$, for all $x \in M'$, by restricting a representative
coordinate chart to $\BR^n \times 0$ and
composing with $\pi$.  This association is in fact a $G$-equivariant
map $\mathcal{S} \rightarrow \mathcal{F}^{(k)}N$.  By the elementary
fact that freeness and properness of an action pulls back by
equivariant maps, we see that $G$ acts freely and properly on
$\mathcal{S}$.

Recall that $\mathcal{S}$ is a closed subbundle of $\mathcal{F}^{(k)}M'$.
The group $\GL^{(k)}(m,\BR)$ acts transitively
on fibers of $\mathcal{F}^{(k)}M'$, commuting with the $G$-action.
The stabilizer in $G$ of $\xi \in \mathcal{F}^{(k)}M'$ is thus equal
the stabilizer of $\xi.h$ for any $h \in \GL^{(k)}(m,\BR)$.  Since $G$
acts freely on $\mathcal{S}$, it acts freely on all
of $\mathcal{F}^{(k)}M'$.

Let $K \subset \mathcal{F}^{(k)}M'$ be a compact subset and consider
$G_K$, comprising all $g \in G$ with $g.K \cap K \neq \emptyset$.  Let
$\bar{K}$ be the projection of $K$ to $M'$ and cover $\bar{K}$ with
finitely-many compact sets $\bar{U}_i$ over which the bundle
$\mathcal{F}^{(k)}M'$ is trivializable.   Let $U_i \subset
\mathcal{S}$ be sections over $\bar{U}_i$.  There are compact subsets
$H_i \subset \GL^{(k)}(m,\BR)$ such that $K \subseteq \cup_{i}
U_i.H_i$.  Now
$$ G_K \subseteq \bigcup_{i,j} G_{U_i.H_i,U_j.H_j}$$
where $G_{A,B}$ comprises the elements $g$ with $g.A \cap B \neq \emptyset$.
These subsets in turn can be expressed
$$ G_{U_i.H_i,U_j.H_j} = G_{U_i,U_j.(H_jH_i^{-1})} = G_{U_i,V}$$
where $V = U_j. (H_jH_i^{-1}) \cap \mathcal{S}$, because $U_i \subset
\mathcal{S}$ and $\mathcal{S}$ is $G$-invariant.  Now because
$\mathcal{S}$ is closed, and by properness of
the $G$-action on $\mathcal{S}$, the set $G_{U_i,V}$ is compact.  We
conclude that $G_K$ is compact, so $G$ acts properly on all of $\mathcal{F}^{(k)}M'$.
\end{proof}

\begin{lemma}
  \label{lem:free.proper.closure}
Let $U \subset M$ equal the closure of its interior $\mathring{U}$, and assume that
$\partial U = D$ is a smooth hypersurface, not necessarily connected.
Let $G$ act smoothly on $M$ leaving $U$ invariant.
For any $k \geq 0$, if $G$ acts freely and properly on
$\mathcal{F}^{(k)} \mathring{U}$ and on $\mathcal{F}^{(k)} D$, then
$G$ acts freely and properly on $\left. \mathcal{F}^{(k)} M \right|_U$.
\end{lemma}

\begin{proof}
  Let $n = \dim M$.
Let $\mathcal{S}$ comprise the $k$-frames in
$\left. \mathcal{F}^{(k)}M \right|_D$ at $x \in D$ arising from
coordinate charts $\varphi$ such that $\varphi(\BR^{n-1} \times 0)$ is
tangent at $x$ to $D$ up to order $k$---in other words, if $F$ is a
defining function for $D$ in a neighborhood of $x$ in $M$, then $F
\circ \varphi$ restricted to $\BR^{n-1} \times 0$ vanishes to order
$k$ at $0$.  Now $\mathcal{S}$ is a closed, $G$-invariant subbundle of
$\left. \mathcal{F}^{(k)}M \right|_D$.  Each $k$-frame in
$\mathcal{S}$ determines a $k$-frame of $D$, and this correspondence
is a $G$-equivariant map from $\mathcal{S}$ to $\mathcal{F}^{(k)}D$.
As in the previous proof, we conclude that $G$ acts freely and
properly on $\mathcal{S}$ and then, using the
$\GL^{(k)}(n,\BR)$-action,
that $G$ acts freely and properly on the entire $\left. \mathcal{F}^{(k)}M \right|_D$.

Note that
  $\mathcal{F}^{(k)}\mathring{U} = \left. \mathcal{F}^{(k)}M
  \right|_{\mathring{U}}$.  Thus $G$ acts freely on $\left. \mathcal{F}^{(k)}
  M \right|_U =  \left. \mathcal{F}^{(k)}M
  \right|_{\mathring{U}} \cup \left. \mathcal{F}^{(k)}M \right|_D$.

  Given a compact subset $K$ of $\left. \mathcal{F}^{(k)}
  M \right|_U$, suppose first that the projection of $K$ to $U$ lies
in $\mathring{U}$.  Then $G_K$ is compact by our assumption on
$\mathcal{F}^{(k)}\mathring{U}$.  Otherwise, the projection of $K$ has
nontrivial intersection with $D$.  Let $K' = K \cap
\left( \left. \mathcal{F}^{(k)}M \right|_D \right)$.  Since the latter set is closed,
$K'$ is also compact.  Then $G_K \subseteq G_{K'}$, which is compact
by properness of the $G$-action on $\left. \mathcal{F}^{(k)}M \right|_D.$
This completes the proof.
  \end{proof}

\begin{lemma}
  \label{lem:free.proper.union}
Let $M = U \cup V$ be a smooth manifold, and $U$ and $V$ closed subsets.  Let $G$
act smoothly on $M$, leaving $U$ and $V$ invariant.  If $G$
acts freely and properly on $\left. \mathcal{F}^{(k)}M \right|_U$ and
$\left. \mathcal{F}^{(k)} M \right|_V$, then $G$ acts freely and properly on
$\mathcal{F}^{(k)}M$, for any $k \geq 0$.
\end{lemma}

\begin{proof}
This proof proceeds easily from the decomposition of
$\mathcal{F}^{(k)}M$ into closed sets $\left. \mathcal{F}^{(k)}M
\right|_U$ and $\left. \mathcal{F}^{(k)}M \right|_V$.
\end{proof}

Here is the proof of Proposition \ref{prop:invt.rgs}.

\begin{proof}
 Projective structures are $2$-rigid
geometric structures in
the sense of Gromov (see \cite[Ch 4]{kobayashi.transf}).
  Thus $G$ acts freely
  and properly on $\mathcal{F}^{(2)}N$ for $N = {\bf S}^{n-1}$ or
  ${\bf RP}^{n-1}$.  The actions in construction I have
  $G$-equivariant maps to $G/Q = {\bf RP}^{n-1}$.  They satisfy the
 conclusion of the proposition by Lemma
 \ref{lem:free.proper.submersion}.

 Now assume $M$ arises from construction II.  The subset $\mathring{V} = G \times_{Q^0}
  (z_-,z_+)$ is open and $G$-invariant, where $z_-$ and $z_+$ are the first and
  last zeroes of the vector field $X$ inside $(-1,1)$; if $z_- = z_+$,
  then take $\mathring{V} = \emptyset$.  First assume $\mathring{V}
  \neq \emptyset$.   There is a
  $G$-equivariant submersion $\mathring{V} \rightarrow G/Q^0 = {\bf S}^{n-1}$,
  so $G$ acts freely and properly on $\mathcal{F}^{(2)} \mathring{V}$ by Lemma
  \ref{lem:free.proper.submersion}.
  Let $V$ be the closure, with boundary a union of two hyperspheres.
  As $G$ preserves
  a projective structure on $\partial V$, it is free and proper on
  $\mathcal{F}^{(2)} \partial V$.  Lemma \ref{lem:free.proper.closure}
gives that $G$ is free and proper on $\left. \mathcal{F}^{(2)} M
\right|_V$.

  Let $U_-$ be the closure in $M$ of the copy of $\BR^n$ glued along
  $\ell_0$ to $(-1,z_-)$.  This is a closed disk, with the standard
  linear $G$-action on the interior $\mathring{U}_-$,
  the standard action on the
  boundary ${\bf
    S}^{n-1}$-fiber over $z_-$, and the normal action to this boundary
  determined by the germ of $X$ at $z_-$.  Regardless of this germ,
  $G$ is free and proper on $\left. \mathcal{F}^{(2)} M
  \right|_{U_-}$.  Indeed, the $G$-action on $\mathring{U}_-$ is affine, so it is
  free and proper on $\mathcal{F}^{(1)}\mathring{U}_-$ and thus also on
  $\mathcal{F}^{(2)}\mathring{U}_-$.  Then Lemma
  \ref{lem:free.proper.closure} applies to give the desired
  conclusion.
Similarly for $U_+$, the closure in $M$ of the copy
  of $\BR^n$ glued along $\ell_0$ to $(z_+,1)$, the $G$-action
  on $\left. \mathcal{F}^{(2)} M
  \right|_{U_+}$ is free and proper.  Still assuming $\mathring{V}
  \neq \emptyset$, two applications of Lemma
  \ref{lem:free.proper.union} on $M = U_- \cup V \cup U_+$, lead to
  the conclusion that $G$ acts freely and
  properly on $\mathcal{F}^{(2)} M$.

  If $\mathring{V} = \emptyset$, then Lemma
  \ref{lem:free.proper.union} applies to $M = U_- \cup U_+$ to yield
  the same conclusion.
We have proved 2-rigidity of the sphere actions in
  construction II.
  The ${\bf RP}^n$-actions are double-covered by sphere actions, so
  the conclusion applies to them, as well.
  \end{proof}

  Recall from Section \ref{subsec:intro_rgs} of the introduction that
  Benveniste--Fisher proved in \cite{benveniste.fisher.no.rgs} nonexistence of an invariant rigid
  geometric structure of algebraic type for certain exotic $\SL(n,\BZ)$-actions
  on ${\bf T}^n$ constructed by Katok--Lewis in \cite{katok.lewis.blowup}.
  That proof relied on the affine local action of ${\bf R}^n$
  on ${\bf T}^n \backslash \{ 0 \}$, which is not available for
  the $\SL(n,\BR)$-actions of Theorems \ref{thm:no_fixed_points} and
  \ref{thm:with_fixed_points}.  We formulate here a variant of
  Question \ref{qtn:rgs.algebraic}:

  \begin{qtn}
    Which smooth $\SL(n,\BR)$-actions on closed $n$-manifolds preserve
    a rigid geometric structure of algebraic type?
    \end{qtn}

    We identify some in the next section, and we expect that these are
    the only ones.

    \subsection{No invariant projective structure for nonstandard
      actions}
    \label{subsec:no_projective}

A projective structure on a manifold $M^n$
determines a canonical Cartan geometry modeled on ${\bf RP}^n$,
comprising a principal $Q_{n+1}$-bundle over $M$ equipped with an
$\mathfrak{sl}(n+1,\BR)$-valued $1$-form satisfying three axioms (see
\cite[Thm 3.8]{sharpe} or \cite[Thm 4.2]{kobayashi.transf}).  Here
$Q_{n+1} < \SL(n+1,\BR)$ is the maximal parabolic subgroup stabilizing a
line of $\BR^{n+1}$ in the standard representation, as usual.  A
consequence is that the isotropy in the group of projective
transformations at any point of
$M$ admits an injective homomorphism to $Q_{n+1}$.

If the projective Weyl curvature of a projective
structure on $M^n$ vanishes, then $M$ is \emph{projectively flat} and has a $(\mbox{PSL}(n+1,\BR),{\bf
  RP}^n)$-structure.  Such a structure corresponds to a projective map
$\delta: \widetilde{M} \rightarrow {\bf S}^n$ called
the \emph{developing map}, a local diffeomorphism, equivariant with respect to a
\emph{holonomy homomorphism} $\rho : \pi_1(M) \rightarrow
\SL(n+1,\BR)$.  See \cite{thurston.3d.book, ot.proj.book} for more about these structures.

An $n$-dimensional \emph{Hopf manifold} is a compact quotient $(\BR^n
\backslash \{ 0 \})/\Lambda$ for $\Lambda$ a lattice in the group of
scalars $\BR^*$, such as
$\Lambda = \{ 2^k \cdot \mbox{Id}_n \ : \ k \in \BZ \}$.  The transitive
$\SL(n,\BR)$-action preserves the flat
connection on these spaces inherited from $\BR^n$.  Note that the
connection on the quotient is
not the Levi-Civita connection of any metric, because it is not unimodular.  These actions are
projective.
Hopf manifolds arise from Construction I with $X$ a nonvanishing vector field on
$\Sigma^0 = S^1$.

The
standard action of $\SL(n,\BR)$ on ${\bf S}^n$ preserves the standard
projective structure, which can be viewed as a projective
compactification of two copies of ${\bf R}^n$ by ${\bf S}^{n-1}$.
It arises from Construction II from a vector field $X$ with a single
zero in $(-1,1)$ of order one and derivative $-1$.

\begin{thm}
  \label{thm:no_projective}
Let $G$ be locally isomorphic to $\SL(n,\BR)$, acting smoothly on a compact $n$-manifold $M$,
preserving a projective structure $[ \nabla ]$.  Then $(M,[\nabla])$
is equivalent to
\begin{itemize}
\item ${\bf S}^n$ or ${\bf RP}^n$ with the standard projective structure
\item a Hopf manifold, diffeomorphic to a flat
  circle bundle over ${\bf RP}^{n-1}$ or ${\bf S}^{n-1}$ with trivial or ${\bf
    Z}_2$ monodromy.
  \end{itemize}
  \end{thm}

  \begin{proof}
   If the $G$-action is transitive and $n >
    4$, then it is type I and $M$ is a quotient of ${\bf R}^n
\backslash \{ 0 \}$ by a cocompact, discrete group of scalar matrices,
by Theorem \ref{thm:orbits}---a Hopf manifold.  For $n=4$, if $M$ is
not a Hopf manifold then it equals the Grassmannian $\mathcal{F}_2^4$,
up to finite covers.  Similarly, if $n=3$ and $M$ is not a Hopf
manifold, then it equals the flag variety $\mathcal{F}_{1,2}^3$, up to
finite covers.  We will show that these homogeneous spaces do not have
an invariant projective structure.

The stabilizer $P_2^4 < \SL(4,\BR)$ of a point of $\mathcal{F}_2^4$ is a semidirect product
$S(\GL(2,\BR) \times \GL(2,\BR)) \ltimes U$, where $U$ is isomorphic to the abelian
group of linear endomorphisms of $\BR^2$.  The fact that $\Ad U$ is
trivial on $\mathfrak{sl}(4,\BR)/\mathfrak{p}_2^4$ corresponds to the
differentials of all elements of $U$ being trivial at the $P_2^4$-fixed point in $\mathcal{F}_2^4$.
The stabilizer $P_{1,2}^3$ of a point of $\mathcal{F}_{1,2}^3$ is a
semidirect product $(\BR^*)^2 \ltimes N$ with $N$ isomorphic to the
$3$-dimensional Heisenberg group.  The center $Z(N)$ acts trivially
via $\Ad$ on $\mathfrak{sl}(3,\BR)/\mathfrak{p}_{1,2}^3$, which
corresponds to its differential being trivial on
$\mathcal{F}_{1,2}^3$ at the $P_{1,2}^3$-fixed point.  These
  projective transformations with trivial differential are called
  \emph{strongly essential} (see \cite{cap.me.proj.conf}, \cite{mn.1graded}).

  Nagano--Ochiai proved that if there is a strongly essential
  $1$-parameter subgroup of the stabilizer of a point $p \in M$ in the
  projective group, then a neighborhood of $p$ in $M$ is projecively
  flat \cite[Lem 5.6]{nagano.ochiai.proj}.  This gives a projective
  local diffeomorphism from a neighborhood of any point of $\mathcal{F}_2^4$ or $\mathcal{F}_{1,2}^3$ to an open
  subset of ${\bf S}^4$ or ${\bf S}^3$, respectively.  All local
  projective transformations of ${\bf S}^n$ are restrictions of
  elements of $\mbox{SL}(n+1,\BR)$ (see \cite[Thm 5.5.2]{sharpe}).  By
  transitivity of the projective $G$-actions, the developing map would be a
  $G$-equivariant projective embedding of $\mathcal{F}_2^4$ or
  $\mathcal{F}_{1,2}^3$---or a finite cover---into ${\bf S}^4$ or ${\bf S}^3$,
  respectively, for some monomorphism $G \rightarrow \SL(n+1,\BR)$.  There is no closed, $n$-dimensional orbit of
  $\SL(n,\BR)$ in the projective action on ${\bf S}^n$, so this is a contradiction.

 Next assume the $G$-action on $M$ is not transitive.  By Theorems
\ref{thm:no_fixed_points} and \ref{thm:with_fixed_points}, the action
arises from Construction I or II.  In either case there is a
closed $(n-1)$-dimensional orbit $O$, equivalent to ${\bf S}^{n-1}$ or ${\bf
  RP}^{n-1}$ by Theorem \ref{thm:orbits}.

There is a $1$-parameter group of strongly essential
projective transformations in this case, too.
Let $p_0$ be a $Q^0$-fixed
point in $O$.  The unipotent radical $U$ of $Q^0$ is in the kernel of
the differential along $O$ at $p_0$.  The $Q^0$-invariant curve of
$C^0$-fixed points runs through $p_0$ transversal to $O$; denote it
$\Sigma$.  The $Q^0$-action on $\Sigma$ factors through $\nu^0$, so
it is pointwise fixed by $U$.  Thus $U$ is in the kernel of the full
differential at $p_0$.

Now \cite[Lem 5.6]{nagano.ochiai.proj} again says that the
projective structure on $M$ is flat in a neighborhood of $p_0$.  By
$G$-invariance of the projective structure, it is projectively flat in a neighborhood $V$ of $O$.
The developing map $\delta : \widetilde{V} \rightarrow {\bf S}^n$ is a
local diffeomorphism.  Here $\widetilde{V}$ can be assumed diffeomorphic to ${\bf
  S}^{n-1} \times (-\epsilon,\epsilon)$.  The vector fields generating the $G$-action on
$\widetilde{V}$ are conjugated by $\delta$ to projective vector fields
on ${\bf S}^n$, forming a subalgebra of $\mathfrak{sl}(n+1,\BR)$ isomorphic to
$\mathfrak{sl}(n,\BR)$.  Let $G'$ be the corresponding subgroup of $\SL(n+1,\BR)$.
The developing image of $\widetilde{O}$ is an $(n-1)$-dimensional
orbit $O'$ of $G' \cong \SL(n,\BR)$.  Up to a projective
transformation of ${\bf S}^n$, it must be the hypersphere ${\bf S}^{n-1} \subset {\bf S}^n$.
Now $\delta$ restricts to an equivariant diffeomorphism
$\widetilde{O} \rightarrow O'$.

Next, $V' = \delta(\widetilde{V})$ is diffeomorphic to ${\bf  S}^{n-1} \times (-\epsilon,\epsilon)$; moreover $\delta$ is a diffeomorphism $\widetilde{V} \rightarrow V'$.
The saturation $G.\widetilde{V}$ is projectively flat, and
its developing image is the saturation $G'.V'$.  The latter set is the
complement of the two $\SL(n,\BR)$-fixed points in ${\bf S}^n$.

Replace $V$ with $G.V$ and $V'$ with $G'.V'$, and consider a point $p$ on the boundary of
$V$.  The orbit $N = G.p$ is necessarily closed.  If it is ${\bf S}^{n-1}$
or ${\bf RP}^{n-1}$, then the argument above implies that
$\widetilde{N}$ develops onto a hypersphere in ${\bf S}^n$.  Because
$\delta$ is a local diffeomorphism and maps $\widetilde{V}$ to $V'$
diffeomorphically, the
image $\delta(\widetilde{N})$ must be on the boundary of $V'$, which
comprises only points.  We conclude that the boundary of $V$ comprises
$G$-fixed points.  In Construction II there are at most two
$G$-fixed points.  If $M$ has one, then it is equivalent to ${\bf
  RP}^n$ with the standard action, and if $M$ has two, then it is
equivalent to ${\bf S}^n$ with the standard action.
    \end{proof}


\bibliographystyle{amsplain}
\bibliography{karinsrefs}

\end{document}